\documentclass[10t,a4paper]{amsart}
\usepackage[utf8]{inputenc}
\usepackage[T1]{fontenc}

\pdfoutput=1

\usepackage[foot]{amsaddr}

\title[Deciding if a hyperbolic group splits over a given subgroup]{Deciding if a hyperbolic group splits over a given quasiconvex subgroup}
\author{Joseph Paul MacManus}

\address{Mathematical Institute, Andrew Wiles Building, Observatory Quarter, University of Oxford, Oxford, OX2 6GG, UK}
\email{macmanus@maths.ox.ac.uk}
\date{First draft: 22 October 2022. Final version: 28 May 2024.}

\usepackage{amsmath}
\usepackage{amsthm}
\usepackage{amssymb}

\usepackage{graphicx}
\usepackage[dvipsnames]{xcolor}
\usepackage{tikz}
\usetikzlibrary{patterns}
\usepackage[colorlinks]{hyperref} % should usually be imported last
\hypersetup{
    linkcolor=black,
    citecolor=blue,
}

\DeclareMathOperator{\Comm}{Comm}
\DeclareMathOperator{\Lab}{Lab}

\newcommand{\Z}{\mathbb{Z}}
\newcommand{\N}{\mathbb{N}}

\newcommand{\set}[2]{\{#1  :  #2\}}

\newcommand{\shad}{\mathcal{S}}

\newcommand{\into}{\hookrightarrow}

\newtheorem{theorem}{Theorem}[section]% theorem counter resets every \subsection
\newtheorem{proposition}[theorem]{Proposition}
\newtheorem{lemma}[theorem]{Lemma}
\newtheorem{corollary}[theorem]{Corollary}

\theoremstyle{definition}
\newtheorem{definition}[theorem]{Definition}
\newtheorem{example}[theorem]{Example}
\newtheorem{remark}[theorem]{Remark}

\numberwithin{equation}{section}

\begin{document}

\begin{abstract}
We present an algorithm which decides whether a given quasiconvex residually finite subgroup $H$ of a hyperbolic group $G$ is associated with a splitting. The methods developed also provide algorithms for computing the number of filtered ends $\tilde e(G,H)$ of $H$ in $G$ under certain hypotheses, and give a new straightforward algorithm for computing the number of ends $e(G,H)$ of the Schreier graph of $H$. Our techniques extend those of Barrett via the use of labelled digraphs, the languages of which encode information on the connectivity of $\partial G - \Lambda H$. 
\end{abstract}

\maketitle

% \newpage

% \tableofcontents

% \newpage

\section{Introduction}

The study of decision problems within group theory is almost as old as the definition of an abstract group itself, dating back to Dehn's classical word, conjugacy, and isomorphism problems. The classical theorems of Novikov--Boone \cite{novikov1955algorithmic, boone1958word} and Adian--Rabin \cite{adian1957unsolvability, rabin1958recursive} state that these problems --- among many, many others --- turn out to be undecidable in the class of finitely presented groups. Thus, if we wish to search for effective solutions to group theoretical problems, one must restrict their scope to some ``nice'' subclass of groups. 
One such subclass is the class of hyperbolic groups. Introduced by Gromov in his landmark essay \cite{gromov1987hyperbolic}, these are groups whose Cayley graphs possess geometric properties reminiscent of negative curvature. Indeed, within this class many problems become decidable. For example, the class of hyperbolic groups has uniformly solvable word and conjugacy problems \cite[ch.~III.H]{bridson2013metric}, and more recently it was shown that one can distinguish isomorphism classes of hyperbolic groups \cite{sela1995isomorphism, dahmani2011isomorphism}.

Another algorithmic problem which has received attention in recent years is that of splitting detection. In the language of Bass--Serre theory \cite{serre2002trees}, a group $G$ \textit{splits} over a subgroup $H$ if $G$ admits a minimal simplicial action on a tree $T$ without inversions, and $H$ stabilises an edge in this action. This problem could be traced back to the algorithm of Jaco--Oertel \cite{jaco1984algorithm} which decides if a given closed irreducible 3-manifold $M$ is Haken, or equivalently if $\pi_1(M)$ splits over an infinite surface group. 
Splittings over finite subgroups are called  \textit{finite splittings}, and a celebrated theorem of Stallings \cite{stallings1968torsion, stallings1972group} states that a finitely generated group admits a finite splitting if and only if it has more than one geometric end (cf. \cite{scott1979topological}). This provides a powerful link between the coarse geometry of a group and its splitting properties. Returning to the realm of hyperbolic groups, we have the following unpublished result due to Gerasimov \cite{gerasimov}.

\begin{theorem}[Gerasimov]
There is an algorithm which, upon input of a presentation of a hyperbolic group $G$, will compute the number of ends of $G$.
\end{theorem}

In particular, one can effectively detect finite splittings of hyperbolic groups. This result was later generalised by Diao--Feighn to finite graphs of finitely generated free groups \cite{diao2005grushko}, by Dahmani--Groves \cite{dahmani2008detecting} to relatively hyperbolic groups, and by Touikan \cite{touikan2018detecting} to finitely presented groups with a solvable word problem and no 2-torsion. Also worthy of mention is the algorithm by Jaco--Letscher--Rubinstein for computing the prime decomposition of a closed orientable 3-manifold \cite{jaco2002algorithms}, as well as the classical description of the Grushko decomposition of a one-relator groups \cite[Prop.~II.5.13]{lyndon1977combinatorial}.

Finite splittings aside, the next logical step is to detect splittings over two-ended (i.e. virtually cyclic) subgroups. This was achieved for (relatively) hyperbolic groups independently by Barrett \cite{barrett2018computing} and Touikan \cite{touikan2018detecting} using quite distinct approaches. Note that Touikan's algorithm here only applies in the torsion-free case. 

\begin{theorem}[Barrett, Touikan]
There is an algorithm which, upon input of a presentation of a hyperbolic group $G$, will decide if $G$ splits over a two-ended subgroup.
\end{theorem}

The algorithm of Barrett, which is of particular interest to us, makes use of Bowditch's deep theorem on two-ended splittings of hyperbolic groups \cite{bowditch1998cut}. This theorem states that a one-ended hyperbolic group $G$ which is not virtually Fuchsian will admit such a splitting if and only if $\partial G$ contains a cut pair. 
In fact, Barrett applies this result to effectively construct Bowditch's canonical JSJ decomposition of a hyperbolic group.

In this paper we will aim to expand on the techniques of \cite{barrett2018computing}, and apply them to larger splittings. We will restrict our attention to \textit{quasiconvex} subgroups, i.e. those subgroups whose inclusion maps are quasi-isometric embeddings, since distorted subgroups exhibit global geometry which is harder to understand on a local scale. 
If a group $G$ splits over a subgroup commensurable with $H$, we say $H$ is \textit{associated with a splitting}. Finding sufficient conditions for a subgroup to be associated to a splitting is a problem which has received a great amount of interest (e.g. \cite{scott1998symmetry, scott2000splittings, niblo2002singularity, scott2003regular, niblo2005minimal}). Applying the results of \cite{niblo2005minimal} to the setting of quasiconvex subgroups of hyperbolic groups, we are able to prove the following decidability result. 

\begin{theorem}[cf.~\ref{cor:detect-split-rf}]
There is an algorithm which takes in as input a one-ended hyperbolic group $G$ and generators of a quasiconvex, residually finite subgroup $H$. 
This algorithm will then decide if $H$ is associated with a splitting, and will output such a splitting if one exists. 
\end{theorem}

It is possible to somewhat weaken the residual finiteness assumption placed on $H$ in the theorem above and give a more general (but more involved) result. We will postpone this more technical statement until Section~\ref{sec:spltting-algs} (see Theorems~\ref{thm:decide-splitting-finitecoends},~\ref{thm:decide-splittings-lonely}). 

\medskip

% The reason for the technicality relating to (non-)lonely subgroups is that while we may find no splitting over $H$ directly, it is possible that $G$ may split over a finite index subgroup $K$ of $H$, where $H$ is a vertex group of this splitting. If neither of the above two hypotheses are satisfied then it is seemingly very difficult to detect splittings of the aforementioned form. Further discussion of this issue is given in Section~\ref{sec:spltting-algs}. An immediate corollary of our result, which is nicer to state, is the following. 

% \begin{corollary}
% There is an algorithm which, upon input of a presentation of a one-ended hyperbolic group $G$ and generators of a quasiconvex residually finite subgroup $H$, will decide if $H$ is associated with a splitting. 
% \end{corollary}

% It is unknown whether there exists a hyperbolic group which is not residually finite, so this corollary applies in ``most'' cases, in some sense. Note also that we could also easily drop the one-ended hypothesis by passing to a maximal finite splitting of $G$. 
% \medskip

In light of Stallings' Theorem, it is a natural generalisation to define the number of ends of a pair of groups $(G,H)$ where $H \leq G$. This definition was first introduced by Houghton \cite{houghton1974ends} and later explored in more depth in the context of discrete groups by Scott \cite{scott1977ends}. The number of ends of the pair $(G,H)$, denoted $e(G,H)$, can be identified with the number of geometric ends of the quotient of the Cayley graph of $G$ by the left action of $H$. This quotient graph is sometimes called the \textit{coset graph} or \textit{Schreier graph} of $H$. It is not hard to show that if $G$ splits over $H$ then $e(G,H) \geq 2$, but the converse does not hold. Our methods give a new proof of the following theorem, originally due to Vonseel \cite{vonseel2018ends}. 

\begin{theorem}[Vonseel, cf.~\ref{thm:decide-coends-scott}]
There is an algorithm which, upon input of a one-ended hyperbolic group $G$ and generators of a quasiconvex subgroup $H$, will output $e(G,H)$. 
\end{theorem}

There is a competing notion of ``ends'' of a pair of a groups which goes by several names in the literature. This idea was considered independently by Bowditch \cite{bowditch2002splittings}, Kropholler--Roller \cite{kropholler1989relative}, and Geoghegan \cite{geoghegan2007topological}, who refer to this invariant as \textit{coends}, \textit{relative ends}, and \textit{filtered ends} respectively. See \cite[ch.~2]{scott2003regular} for a discussion on the equivalence of these three definitions. In this paper we will adopt the terminology and notation of Geoghegan, and denote the number of filtered ends of the pair $(G,H)$ by $\tilde e(G,H)$. This value appears to be more resilient to calculation without extra hypotheses, but nonetheless we have some partial results. Recall that the \textit{generalised word problem} for a finitely generated group $H$ is the problem of, given words $w_0, \ldots w_n$ in the generators of $H$, deciding whether $w_0 \in \langle w_1, \ldots w_n\rangle_H$. We then have the following statement. 

\begin{theorem}[cf.~\ref{thm:decide-limsetcomp-discon},~\ref{thm:decide-filtered-ends-finite}]
There is an algorithm which takes in as input a one-ended hyperbolic group $G$, and generators of a quasiconvex subgroup $H$. This algorithm will terminate if and only if $\tilde e(G,H)$ is finite, and if it terminates will output the value of $\tilde e(G,H)$. 

Furthermore, if one is also given a solution to the generalised word problem for $H$, then there is an algorithm which decides whether $\tilde e(G,H) \geq N$ for any given $N \geq 0$. 
\end{theorem}

We remark in Section~\ref{sec:ends-revisited} that $\tilde e(G,H)$ can be identified with the number of components of $\partial G - \Lambda H$. Thus, the above algorithm allows us to decide if $\partial G - \Lambda H$ is disconnected. 
It's also worth noting that if we know a priori that $\tilde e(G,H)$ is finite, for example if $H$ is two-ended, then $\tilde e(G,H)$ is fully computable. We will also see that the value of $\tilde e(G,H)$ is computable if $H$ is free. It does not seem possible to decide in general if $\tilde e(G,H) = \infty$ using our machinery for an arbitrary quasiconvex subgroup, without assuming further hypotheses. We discuss this limitation in Section~\ref{sec:counting-ends}.

\subsection*{Acknowledgements}

My thanks go to Panos Papazoglou for suggesting this problem, and for many helpful discussions. I'm also grateful to Sam Hughes and Ric Wade for their detailed feedback, to Michah Sageev and Henry Wilton for fruitful exchanges, and to Thomas Delzant for pointing me towards Vonseel's work. Finally, I thank the referees for their helpful suggestions.

\section{Preliminaries}

In this section we recall the basic notions and tools we require. We begin with a look at almost invariant sets and (filtered) ends of pairs, before briefly turning towards hyperbolic groups and their quasiconvex splittings. Throughout this paper we will assume a working knowledge of Bass--Serre theory, a good reference for which is \cite{serre2002trees}. 

\subsection{Almost invariant subsets and (filtered) ends}

We will need the idea of an almost invariant subset. A very good introduction to the upcoming definitions can be found in \cite{scott1998symmetry}, which features many helpful examples. The reader should note however that this paper contains an error, a correction of which can be found in \cite{scott2003regular}.

First, some notation. In what follows, $G$ will be a finitely generated group and $H \leq G$ a finitely generated subgroup. If $Z$ is a set upon which $H$ acts on the left, then denote by $H \backslash Z$ the quotient of $Z$ by this action. 

\begin{definition}
Let $U$ and $V$ be two sets. Denote by $U \triangle V$ the \textit{symmetric difference} of $U$ and $V$, defined as 
$$
U \triangle V := (U - V) \cup (V - U).
$$
We say that two sets $U$, $V$ are \textit{almost equal} if $U \triangle V$ is finite. 
\end{definition}

\begin{definition}
Let $G$ act on the right on a set $Z$. We say that $U \subset Z$ is \textit{almost invariant} if for all $g \in G$, $Ug$ is almost equal to $U$.
\end{definition}

\begin{definition}
We say a subset $U\subset G$ is \textit{$H$-finite} or \textit{small}, if $U$ projects to a finite subset of $H \backslash G$. If $U$ is not $H$-finite, then we say $U$ is \textit{$H$-infinite}, or \textit{large}. 
\end{definition}

To ease notation, given a subset $X \subset G$ we will write $X^\ast := G - X$. 

\begin{definition}
We say that a subset $X \subset G$ is $H$-\textit{almost invariant} if it is invariant under the left action of $H$, and $H\backslash X$ is almost invariant under the right action of $G$ on $H \backslash G$. 
We say that $X$ is \textit{non-trivial} if both $X$ and $X^\ast$ are $H$-infinite. 

Let $X$ and $Y$ be two non-trivial $H$-almost invariant subsets of $G$. We say that $X$ and $Y$ are \textit{equivalent}, if $X \triangle Y$ is $H$-finite.
\end{definition}

\begin{definition}\label{def:crossings}
Let $X$ be an $H$-almost invariant subset. Given $g \in G$, we say that $gX$ \textit{crosses} $X$ if all of 
$$
gX \cap X, \ \ gX \cap X^\ast,  \ \ gX^\ast \cap X,  \ \ gX^\ast \cap X^\ast
$$ 
are large. 
If there exists $g \in G$ such that $gX$ crosses $X$ then we say that $X$ \textit{crosses itself}.
If $X$ does not cross itself, we say it is \textit{almost nested}. If one of the above intersections is empty, we say $X$ is \textit{nested}. 
\end{definition}

It is easy to see that if $X$ and $Y$ are equivalent $H$-almost invariant sets, then $X$ crosses itself if and only if $Y$ crosses itself. 

\begin{example}\label{eg:splitting-nested}
Suppose a group $G$ splits as an amalgam or HNN extension over a subgroup $H$. Then one can construct a non-trivial nested $H$-almost invariant subset $X \subset G$ as follows. Let $T$ be the Bass-Serre tree of this splitting. We now construct a $G$-equivariant map $\phi : G \to VT$, where $VT$ denotes the set of vertices of $T$. Given $1 \in G$, set $\phi(1)$ arbitrarily to some $ w\in VT$. For each $g \in G$, set $\phi(g) = g\phi(1)$. Since $G$ acts upon itself freely and transitively, $\phi$ is well defined for all $g \in G$. 

Given $\phi$ as above, let $e \in ET$ be the edge stabilised by $H$ with endpoints $u, v$. Deleting the interior of $e$ separates $T$ into two components, $T_u$ and $T_v$ containing $u$ and $v$ respectively. Set $X = \phi^{-1} (VT_u)$, then it is a simple exercise to check that $X$ is a non-trivial nested $H$-almost invariant subset of $G$. 
\end{example}

We can now state the following key theorem due to Scott--Swarup \cite{scott2000splittings}, which is in some sense a converse to Example~\ref{eg:splitting-nested}.  Recall that a subgroup $H$ of $G$ is said to be \textit{associated to a splitting} if $G$ splits over a subgroup commensurable with $H$.

\begin{theorem}[{\cite[Thm.~2.8]{scott2000splittings}}]\label{thm:almost-nest-split}
Let $G$ be a finitely generated group, $H$ a finitely generated subgroup, and $X$ an $H$-almost invariant subset of $G$. Suppose that $X$ is almost nested, then $H$ is associated to a splitting. 
\end{theorem}

There is a generalisation of the above, which will be important to us. Firstly, we must further loosen our requirements for nesting. Denote by $\Comm_G(H)$ the commensurator of $H$ in $G$. That is, 
$$
\Comm_G(H) = \{g \in G : |H : H\cap H^g| < \infty, \ |H^g : H\cap H^g| < \infty\}.
$$
Then we have the following definition.

\begin{definition}
Let $X \subset G$ be $H$-almost invariant. We say that $X$ is \textit{semi-nested} if $\{ g \in G  : \textrm{$gX$ crosses $X$} \}$ is contained in $\Comm_G(H)$. 
\end{definition}

Informally, we relax our definition to allow crossings of $X$ by $gX$ on the condition that $gH$ is ``very close'' to $H$. 
We then have the following useful result, due to Niblo--Sageev--Scott--Swarup \cite{niblo2005minimal}, which says that this relaxation still produces splittings. 

\begin{theorem}[{\cite[Thm.~4.2]{niblo2005minimal}}]\label{thm:nsss}
Let $G$ be a finitely generated group and $H$ a finitely generated subgroup. Suppose that there exists a non-trivial $H$-almost invariant subset $X \subset G$ which is semi-nested. Then $G$ splits over a subgroup commensurable with $H$. 
\end{theorem}

This idea of ``crossings'' of almost invariant sets is much more rich than what is presented here, and pertains to the idea of ``compatible'' splittings. 
The interested reader should consult \cite{scott2000splittings}, which features many helpful examples, as a starting point.
\medskip

There is a natural way to ``count'' these $H$-almost invariant subsets, which provides a useful integer invariant of the subgroup.  
Let $\mathcal P(H\backslash G)$ denote the power set of $H\backslash G$. Let $\mathcal F(H \backslash G)$, denote the set of finite subsets. Under the operation of symmetric difference $\triangle$, $\mathcal P(H \backslash G)$ can be seen as a $\Z_2$-vector space, and $\mathcal F( H \backslash G)$ a subspace. The quotient space $\mathcal E(H \backslash G) := \mathcal P(H\backslash G) / \mathcal F(H \backslash G)$ can be identified naturally with the set of $H$-almost invariant sets of $G$, modulo equivalence. 

\begin{definition}
Let $G$ be a group and $H \leq G$. We define the \textit{number of ends of the pair $(G,H)$} as the rank of $\mathcal E (H \backslash G)$ as a $\Z_2$-vector space. 
Denote by $e(G) = e(G,\{1\})$, and say that $e(G)$ is the \textit{number of ends of $G$}. 
\end{definition}

There is also the following characterisation of ends of pairs, which will be helpful later. This result motivates the earlier description that $e(G,H)$ ``counts'' $H$-almost invariant subsets. 

\begin{proposition}[{\cite[Lem.~1.6]{scott1977ends}}]
Let $G$ be a group, $H$ a subgroup, and $n \geq 0$. Then $e(G,H) \geq n$ if and only if there exists a collection of $n$ pairwise disjoint $H$-almost invariant subsets of $G$.
\end{proposition}

There is a more geometric intuition for the above, which can be seen in the coset graph. 

\begin{proposition}[{\cite[Lem.~1.1]{scott1977ends}}]\label{prop:coset-ends}
Let $\Gamma$ be a Cayley graph of $G$, and $H \subset G$. Then $e(G,H)$ is equal to the number of ends of the coset graph $H \backslash \Gamma$. 
\end{proposition}

From the above it is then clear that this definition generalises the standard geometric notion of ends.
This also provides a helpful way of seeing that the number of geometric ends of a finitely generated group is indeed independent of the choice of generating set.
\medskip

There is another competing, but equally interesting notion of ends of a pair of groups, namely the idea of \textit{filtered ends}. This was considered independently
by Geoghegan \cite[Ch.~14]{geoghegan2007topological},  Kropholler--Roller \cite{kropholler1989relative}, and Bowditch \cite{bowditch2002splittings}. 
We now summarise the definition as it appears in \cite{geoghegan2007topological}. We first need the following technical preliminaries relating to filtrations.  

Let $Y$ be a connected, locally finite cell complex. A \textit{filtration} $\mathcal K = \{K_i\}$ of $Y$ is an ascending sequence of subcomplexes $K_1 \subset K_2 \subset \ldots \subset Y$ such that $\bigcup_i K_i = Y$. We say that this filtration is \textit{finite} if each $K_i$ is finite. We call the pair $(Y, \mathcal K)$ a \textit{filtered complex}, and 
a map $f : (Y, \mathcal K) \to (X, \mathcal L)$ between filtered complexes is called a \textit{filtered map} if the following four conditions hold:
\begin{enumerate}
    \item $\forall i$, $\exists j$ such that $f(K_j) \subset L_i$,
    \item $\forall i$, $\exists j$ such that $f(K_i) \subset L_j$,
    \item $\forall i$, $\exists j$ such that $f(Y-K_j) \subset X-L_i$,
    \item $\forall i$, $\exists j$ such that $f(Y-K_i) \subset X-L_j$.
\end{enumerate}
We say that a homotopy $H_t$ between two filtered maps is a \textit{filtered homotopy} if  $H_t$ is a filtered map for each $t$. 
Fix a basepoint $b \in Y$, then a \textit{filtered ray based at b} is a map $\gamma : [0,\infty) \to Y$ with $\gamma(0)=b$, which is filtered with respect to the filtration $\{[0,i]:i \in \mathbb N\}$ of $[0,\infty)$.
We say that two filtered rays based at $b$ are \textit{equivalent} if there is a filtered homotopy between them fixing $b$. A \textit{filtered end} of $(Y, \mathcal K)$ is an equivalence class of filtered rays based at $b$. It is easy to see that the choice of $b$ does not affect the number of filtered ends. 

Let $G$ be a finitely generated group and $H$ a finitely generated subgroup. Let $X_G$ be a Cayley complex for $G$, and let $p : X_G \to H \backslash X_G$ denote quotient map. Choose a finite filtration $\mathcal K = \{K_i\}$ of $H \backslash X_G$, and lift this to a filtration $\mathcal L = \{L_i\}$ of $X_G$, where $L_i = p^{-1}(K_i)$.

\begin{definition}
The \textit{number of filtered ends of the pair} $(G,H)$, denoted by $\tilde e(G,H)$, is defined as the number of filtered ends of the filtered complex $(X_G, \mathcal L)$. 
\end{definition}

Though it seems as though this definition depends on a choice of the filtration $\mathcal K$, it is a helpful fact that it does not --- see \cite[Ch.~14]{geoghegan2007topological} for details. 
While this definition is a little technical, we will see later on that, in the case of quasiconvex subgroups of hyperbolic groups, these filtered ends can be clarified by looking at the Gromov boundary. We conclude this section with two results relating ends of pairs and filtered ends of pairs. 

\begin{proposition}[{\cite[Prop.~14.5.3]{geoghegan2007topological}}]
Let $G$ be a group and $H$ be a finitely generated subgroup. Then $e(G,H) \leq \tilde e(G,H)$.
\end{proposition}

Equality is certainly possible, but not true in general. A simple counterexample is presented in \cite[p.~32]{scott2003regular}, which involves a one-sided essential simple closed curve on a non-orientable surface. We finally conclude with the following, which demonstrates the common thread between the two competing definitions. 

\begin{proposition}[{\cite[Lem.~2.40]{scott2003regular}}]
Let $G$ be a group and $H$ be a finitely generated subgroup of infinite index. Then $\tilde e(G,H) > 1$ if and only if there is some subgroup $K \leq H$ such that $e(G,K) > 1$.
\end{proposition}

\subsection{Hyperbolic groups}

Let $\delta \geq 0$, and recall that a geodesic metric space $X$ is $\delta$\textit{-hyperbolic} if every geodesic triangle in $X$ is $\delta$-slim. That is, if every side of a geodesic triangle is contained in the $\delta$-neighbourhood of the other two sides. We say that a finitely generated group $G$ is \textit{hyperbolic} if some (equivalently, any) Cayley graph of $G$ is $\delta$-hyperbolic for some $\delta \geq 0$.  

For the remainder of this section, let $\Gamma$ be a $\delta$-hyperbolic Cayley graph of the finitely generated group $G$. We recall some basic facts about hyperbolic groups which will be important later. The first is often known as the \textit{visibility} property.

\begin{lemma}\label{lem:visibility}
For every $g, g' \in \Gamma$, there is a geodesic ray $\gamma$ based at $g'$ which passes within $C = 3\delta$ of $g$. 
\end{lemma}

Viewing geodesic rays as ``lines of sight'' from the basepoint, we can imagine this theorem as saying that every point in $\Gamma$ is (nearly) visible from every other point.

The second result we need is the following, commonly known as the \textit{Morse Lemma}. This fact will play a fundamental role in much of the machinery of this paper. Informally, it says that quasi-geodesics must stay uniformly close to geodesics with the same endpoints.   

\begin{lemma}\label{lem:morse}Let $\lambda \geq 1$, $\varepsilon \geq 0$. 
Then there is some computable $D = D(\delta, \lambda, \varepsilon) \geq 0$ such that $(\lambda, \varepsilon)$-quasi-geodesics in $\Gamma$ are contained within the $D$-neighbourhood of any geodesic with the same endpoints. 
\end{lemma}

We now briefly introduce the Gromov boundary of a hyperbolic group. Consider the set $\mathcal R = \mathcal R(\Gamma, 1)$ of geodesic rays in $\Gamma$ based at $1$. Throughout this paper we may abuse notation and identify a ray or path with its image in $\Gamma$. We say that two such rays $\gamma, \gamma' \in \mathcal R$ are \textit{equivalent}, and write $\gamma \sim \gamma'$, if their Hausdorff distance as subsets of $\Gamma$ is finite. Let $\partial \Gamma = \mathcal R / \sim$, and call this set the \textit{Gromov boundary of $\Gamma$.} Given $\gamma \in \mathcal R$, denote by $\gamma(\infty)$ its equivalence class in $\partial \Gamma$. We topologise $\partial \Gamma$ as follows. Given $x,y,z \in \Gamma$, let 
$$
(x\cdot y)_z = \tfrac 1 2 (d(x,z) + d(y,z) - d(x,z))
$$
denote the \textit{Gromov product}. For $p \in \partial \Gamma$, $r \geq 0$, let 
$$
V(p, r) = \set {q \in \partial \Gamma}{\textrm{$\exists \gamma \in p, \gamma' \in q$ such that $\liminf_{t \to \infty} (\gamma(t)\cdot \gamma'(t))_1 \geq r$}}.
$$
Given $p \in \partial \Gamma$, we set $\set{V(p,r)}{r\geq 0}$ as a basis of neighbourhoods about $p$. This topologises $\partial \Gamma$. Informally, two rays are ``close'' in $\partial \Gamma$ if they fellow-travel for a long time. 
It can be seen that, up to homeomorphism, $\partial \Gamma$ does not depend on a choice of basepoint, or even on the choice of Cayley graph $\Gamma$. Thus, we may abuse notation and write $\partial G$. 

One can also define the \textit{sequential boundary} of $\Gamma$. 
Fix a basepoint $z \in \Gamma$. We say that a sequence $(p_n) \in \Gamma^\N$ \textit{tends to infinity} if $\liminf_{n, m \to \infty}(p_n\cdot p_m)_z = \infty$. We say two such sequences $(p_n)$, $(q_m)$ are \textit{equivalent} if $\liminf_{n,m \to \infty} (p_n\cdot q_m)_z = \infty$. 
One can then define $\partial \Gamma$ as the set of equivalence classes of sequences which tend to infinity, and topologise it in a similar way. Given a sequence $(p_n)$ which tends to infinity, denote by $\lim_{n\to \infty} p_n$ its equivalence class in $\partial \Gamma$.  This definition is equivalent to the previous \cite[Prop.~2.14]{kapovich2002boundaries}. 
Other equivalent variations of this construction are possible, and can be found in \cite{kapovich2002boundaries}. 

We can extend the definition of the Gromov product to $\partial \Gamma$, by defining 
$$
(p\cdot q)_z = \sup \left\{ \liminf_{n,m \to \infty} (p_n\cdot  q_m)_z :  (p_n) \in p,  \ (q_m) \in q    \right\} ,
$$
for $p, q \in \partial \Gamma$, $z \in \Gamma$. 
Using this extension there is a natural way to put a metric on $\partial G$, though the metric is not canonical, as it depends on some chosen parameters. In particular, a \textit{visual metric} on $\partial \Gamma$ with parameter $a$ and multiplicative constants $k_1$, $k_2$ is a metric $\rho $ satisfying
$$
k_1 a^{-(p\cdot q)_1} \leq \rho (p,q) \leq k_2 a^{-(p\cdot q)_1} 
$$
for every $p, q \in \partial \Gamma$. See \cite[p.~434]{bridson2013metric} for how to construct such a metric given suitable parameters. 

Note that the number of connected components of the boundary can be identified with the number of ends $e(G)$ of $G$. In particular, the boundary $\partial \Gamma$ is connected if and only if $G$ is one-ended. Thus by Stallings' Theorem for ends of groups (see e.g. \cite{scott1979topological}) we have that $\partial G$ is connected if and only if $G$ admits no finite splitting.  
\medskip

We conclude this section with an important and deep result due to Bestvina--Mess \cite{bestvina1991boundary}, Bowditch \cite{bowditch1998cut}, and Swarup \cite{swarup1996cut}. This result essentially quantifies the connectedness of the boundary and relates it to the local geometry of the Cayley graph. Let $C = 3\delta$ as in Lemma~\ref{lem:visibility}, and let $M = 6C + 2\delta + 3$. We then have the following.

\begin{definition}
We say that $\Gamma$ satisfies $\ddagger_n$ for $n \geq 1$ if for every $R \geq 0$ and every $x, y \in \Gamma$ such that $d(x,1) = d(y,1) = R$ and $d(x,y) \leq M$ we have that there is a path through $\Gamma - B_1(R)$ connecting $x$ to $y$, of length at most $n$. 
\end{definition}

\begin{theorem}[{\cite[Props.~3.1,~3.2]{bestvina1991boundary}}]\label{thm:bestvina-mess}
Let $G$ be a hyperbolic group, with $\delta$-hyperbolic Cayley graph $\Gamma$. Then $\partial G$ is connected if and only if there exists $n \geq 1$ such that $\Gamma$ satisfies $\ddagger_n$. 
\end{theorem}

Note that through the algorithm presented in \cite{dahmani2008detecting}, one can decide if a given hyperbolic group $G$ is one-ended, and if so this algorithm will output $n$ such that $\ddagger_n$ holds in the Cayley graph associated to the given generators of $G$. 
Their algorithm also applies to relatively hyperbolic groups, but we will not use that here.

\subsection{Quasiconvex subgroups and splittings}

Recall that a subset of a geodesic space is termed \textit{convex} if it contains any geodesic between any two of its points. This notion is too precise for the setting of groups, so we must ``quasify'' it. 

Let $Q \geq 0$, and let $G$ be a hyperbolic group with Cayley graph $\Gamma$. Then a subgroup $H \leq G$ is called $Q$\textit{-quasiconvex} if for every $h,h' \in H$, any geodesic path between $h, h'$ in $\Gamma$ is contained in the closed $Q$-neighbourhood of $H$. The quasiconvexity of $H$ does not depend on the choice of Cayley graph for $G$, given that $G$ is hyperbolic, and quasiconvex subgroups of hyperbolic groups are also hyperbolic. The following characterisation of quasiconvexity will be important.

\begin{lemma}\label{lem:qi-embed}
Let $G$ be a hyperbolic group, $H \leq G$ a finitely generated subgroup, and fix word metrics on these groups. Then $H$ is quasiconvex if and only if the inclusion map $H \into G$ is a $(\lambda, \varepsilon)$-quasi-isometric embedding for some $\lambda \geq 1$, $\varepsilon \geq 0$. 

Moreover, given a presentation of $G$ and generators of $H$, if $H$ is quasiconvex then these constants can be found algorithmically.  
\end{lemma}

The above follows immediately from Kapovich's algorithm \cite{kapovich1996detecting} for computing the quasiconvexity constant $Q$, together with \cite[ch.~III.$\Gamma$, Lem.~3.5]{bridson2013metric}.

Quasiconvex subgroups interact with the Gromov boundary in a very controlled way. Informally, if $H \leq G$ is a quasiconvex subgroup, then this induces an inclusion $\partial H \into \partial G$. More precisely, we need the idea of ``limit sets'' in the boundary. 

\begin{definition}
Let $G$ be a hyperbolic group with Cayley graph $\Gamma$, and $H$ some subgroup of $G$. Then define the \textit{limit set} of $H$ as
$$
\Lambda H = \{p \in \partial \Gamma : \textrm{$p = \lim_{n \to \infty} h_n$ for some $h_n \in H$}  \} .
$$
\end{definition}

It is a standard fact that if $H$ is quasiconvex then $\Lambda H \cong \partial H$, and $\Lambda H$ is a closed $H$-invariant subset of $\partial G$ (see e.g. \cite{kapovich2002boundaries}). 
Combining Lemmas~\ref{lem:visibility} and \ref{lem:morse} we immediately arrive at the following, which will play a big role in the next section.

\begin{lemma}\label{lem:eta}
Let $\Gamma$ be a $\delta$-hyperbolic Cayley graph of a hyperbolic group $G$, and let $H$ be a $Q$-quasiconvex subgroup. 
There exists some computable constant $\eta = \eta(\delta, Q) > 0$ with the following property. If $\gamma$ is a geodesic ray in $\Gamma$ based at 1 such that $\gamma(\infty) \in \Lambda H$, then $\gamma$ is contained in the closed $\eta$-neighbourhood of $H$. 
\end{lemma}

Another important fact about quasiconvex subgroups of hyperbolic groups is that they are ``almost'' self-commensurating. More formally we have the following classical result due to Arzhantseva \cite{arzhantseva2001quasiconvex} and Kapovich--Short \cite{kapovich1996greenberg}. 

\begin{theorem}[{\cite[Thm.~2]{arzhantseva2001quasiconvex}}]\label{thm:finite-index-comm}
Let $G$ be a hyperbolic group and $H$ a quasiconvex subgroup. Then $H$ has finite index in $\Comm_G(H)$.  
\end{theorem}

\begin{corollary}\label{cor:maximal-fi}
Let $H$ be a quasiconvex subgroup of a hyperbolic group $G$. Then the commensurator $\Comm_G(H)$ is the unique maximal finite index overgroup of $H$. 
\end{corollary}

\begin{proof}
Let $K \leq G$ be some finite index overgroup of $H$, 
and let $k \in K$. Clearly $H^k \subset K$, so $H^k \cap H \leq H \cap K = H$ and $H^k$ has finite index in $K$. Then
$$
|K : H \cap H^k| = |K:H| \cdot |H : H \cap H^k|, 
$$
so $|H : H \cap H^k| < \infty$. Similarly,  $|H^k : H \cap H^k| < \infty$, and so $k \in \Comm_G(H)$. 
\end{proof}

We conclude this section with a description of the vertex groups of a splitting in which the edge groups are quasiconvex. 
We say that a subgroup is \textit{full} if it does not have any finite index overgroups. 

\begin{proposition}[{\cite[Prop.~1.2]{bowditch1998cut}}]\label{prop:full-vertex-groups}
Let $G$ be a one-ended hyperbolic group, and suppose that $G$ splits over a quasiconvex subgroup as an amalgam or an HNN extension. Then the vertex groups of this splitting are full quasiconvex.
\end{proposition}

\begin{proof}
The fact that they are quasiconvex is proven in \cite{bowditch1998cut}, so we will just show here that they are necessarily also full. 

Let $G$ act simplicially, minimally, and without inversions on a tree $T$ with a single edge orbit, such that the edge stabilisers are quasiconvex. Let $v$ be a vertex of $T$, and let $G_v$ denote the stabiliser of $v$. 
Suppose that $G_v$ is not full, so there is finite index overgroup $K > G_v$. Then the $K$-orbit of $v$ is finite, so there is some finite subtree $\Sigma \subset T$ stabilised by $K$. 
Let $u$ be the geometric center of $\Sigma$. If $u$ is a midpoint of an edge $e$, then it must be the case that the action of $K$ inverts this edge. We assumed that $G$ acts without inversions so this cannot happen, and so $u$ is a vertex of $T$. 
Since $u$ is distinct from $v$ (else $K = G_v$), we see that there is a vertex $w$ adjacent to $v$ such that 
$$
G_v \leq G_w \leq K. 
$$
In particular, $G_v$ has finite index in $G_w$. 

Let $e$ be the edge connecting $v$ to $w$. We split into two cases. Firstly, suppose that this splitting is an amalgam. Then $G_e = G_v \cap G_w = G_v$, but then the given splitting is of the form $G = G_w \ast_{G_v} G_v$, and so $G_w = G$ and $G$ fixes a point on $T$. 

Suppose instead that $G$ acts on $T$ with a single vertex orbit, i.e. that this splitting is an HNN extension. Then there is some $g \in G$ such that $G_w = G_v^g$. Suppose that $G_w > G_v$, then we find the following sequence
$$
G_v < G_v^{g} < G_v^{g^2} < G_v^{g^3} < G_v^{g^4} < \ldots,
$$
where each has finite index in the last. But then $G_v$ has no maximal finite index overgroup, which  contradicts Corollary~\ref{cor:maximal-fi} since $G_v$ is quasiconvex. Thus $G_w = G_v = G_e$. In particular, by translating this picture around we see that the stabiliser of every edge in $T$ is equal to the stabiliser of either of its endpoints. Since $T$ is connected, it follows that $G_v$ actually fixes $T$ pointwise and is in fact the kernel of the action of $G$ on $T$. Thus, $G_v$ is a normal subgroup of $G$. 
Since $G_v$ is quasiconvex, this can only happen if $G_v$ is finite or has finite index in $G$ \cite{kapovich1996greenberg}. By assumption, $G$ is one-ended and so $G_v$ has finite index in $G$. But then $G$ must act on $T$ with a global fixed point. With this contradiction, we conclude that $G_v$ must be full. 
\end{proof}

The following definition will be helpful. 

\begin{definition}[Lonely subgroups]
    Let $G$ be a group, then a subgroup $H \leq G$ is said to be \textit{lonely} if there does not exist another subgroup $H'$ distinct from $H$ such that $H$ is commensurable with $H'$. 
\end{definition} 

We then have the following dichotomy, which will play a central role in our algorithmic results. 

\begin{proposition}\label{prop:quasiconvex-splittings}
Let $G$ be a hyperbolic group and $H$ a quasiconvex subgroup such that either 
\begin{enumerate}
    \item $H \neq \Comm_G(H)$, or
    \item $H$ is lonely.
\end{enumerate}
Then $H$ is associated with a splitting if and only if there is some semi-nested $H$-almost invariant set $X$.
\end{proposition}

\begin{proof}
Applying Theorem~\ref{thm:nsss}, it is clear that if we have a semi-nested $H$-almost invariant set then $H$ is associated to a splitting. 

Conversely, suppose that $G$ splits over a subgroup $H'$ commensurable with $H$. If $H$ is lonely then the result is clear as we necessarily split over $H$, thus there is a nested $H$-almost invariant set as in Example~\ref{eg:splitting-nested}. 
So assume instead that $H \neq \Comm_G(H)$. We now split into two cases. 

Firstly, suppose that $H'$ has infinite index in the vertex group(s) of its splitting. We have by Theorem~\ref{thm:finite-index-comm} that $\Comm_G(H)$ is commensurable with $H'$, so it will be strictly contained in a vertex group of this splitting, up to conjugacy. We can then transform our splitting over $H'$ easily into a splitting over $C := \Comm_G(H)$ via e.g.
$$
G = A \ast_{H'} B = A \ast_{H'} (C \ast_C B) = (A \ast_{H'} C) \ast_C B.
$$
The HNN case is similar. Since we have a non-trivial splitting over $C$, there is an almost nested $C$-almost invariant set by Example~\ref{eg:splitting-nested}. Clearly such a set is also $H$-almost invariant, and so we are done. 

Suppose instead that $H'$ has finite index in one of the vertex stabilisers. 
By Proposition~\ref{prop:full-vertex-groups} we know that the vertex groups of the splitting over $H'$ are full, so by Corollary~\ref{cor:maximal-fi} we deduce that one vertex group of the splitting over $H'$ is precisely $\Comm_G(H)$. 
Let $T$ be the Bass-Serre tree of this splitting. Then $H$ fixes a vertex of $T$, namely the vertex stabilised by $\Comm_G(H)$. 
Let $e \in ET$ be the edge stabilised by $H'$ with endpoints $v, u \in VT$, and suppose without loss of generality that $H$ fixes $v$. Then the $H$-orbit of $e$ is a finite collection of edges abutting $v$. Since $H \neq \Comm_G(H)$ there are multiple $H$-orbits of edges abutting $v$. 
Let $F = \set {he} {h \in H}$. Then deleting the interior of the edges in $F$ separates $T$ into a disjoint collection of subtrees. Let $T_v$ be the subtree which contains $v$, then set $X = \phi^{-1}(VT_v)$, where $\phi : G \to T$ is the map constructed in Example~\ref{eg:splitting-nested}. It is clear that $X$ is a non-trivial $H$-almost invariant set, and 
$$
\{ g \in G  : \textrm{$gX$ crosses $X$} \} \subset G_v = \Comm_G(H),
$$
where $G_v \leq G$ is the stabiliser of $v$.
Thus $X$ is semi-nested by definition. 
\end{proof}

\begin{remark}
    It might seem strange that in the hypothesis of Proposition~\ref{prop:quasiconvex-splittings} we ask that $H$ be distinct from its commensurator. This rules out, for example, maximal cyclic subgroups (which are often associated with splittings). This exclusion can be explained by an example. 
    
    Suppose our group $G$ splits as an amalgam $G= A \ast_C  B$ where $|B : C | < \infty$. Then the vertex group $B$ is associated with a splitting, by definition. However, it is certainly not clear if $G$ contains some semi-nested $B$-invariant subset (it is helpful to check where the construction given in the above proof fails in this case). This technicality can be overcome by replacing our given quasiconvex subgroup with one of its finite index subgroups (if one exists), which explains why it is helpful to assume that our given subgroup is residually finite. 
\end{remark}

\section{Limit Set Complements}

The key aim of this section is to understand the connectivity of $\partial G - \Lambda H$ via local geometry. In particular, we will generalise a toolbox introduced by Barrett in \cite{barrett2018computing}. 

\subsection{Annular neighbourhoods}

The Gromov boundary encodes a wide variety of data about our group, but we cannot access it directly. Thus, we must characterise the presence of the topological features we care about via some kind of equivalent local geometric feature in the Cayley graph.

For what follows we will need to fix some notation. Throughout this section $G$ will be a one-ended hyperbolic group with $\delta$-hyperbolic Cayley graph $\Gamma$, and $H$ will be a $Q$-quasiconvex subgroup with a fixed finite generating set $Y$. Note that given a presentation of a hyperbolic group, one can effectively compute the hyperbolicity constant $\delta$ of the corresponding Cayley graph via \cite{papasoglu1996algorithm}. 
Given $x, y \in \Gamma$, we may denote by $[x,y]$ a choice of geodesic path between $x$ and $y$.
Let $C = 3\delta$ as in  Lemma~\ref{lem:visibility}, $D = D(1,0) $ as in Lemma~\ref{lem:morse}. 
Fix a visual metric $\rho $ on $\partial G$, with parameter $a$ and multiplicative constants $k_1$, $k_2$. 
We will make use of some methods from \cite{bestvina1991boundary}, so recall Theorem~\ref{thm:bestvina-mess} and fix $n\geq 0$ such that $\ddagger_n$ holds in $\Gamma$. 
Finally, choose $\lambda \geq 1$, $\varepsilon \geq 0$ such that the inclusion map $H \into G$ is a $(\lambda, \varepsilon)$-quasi-isometric embedding as in Lemma~\ref{lem:qi-embed}, and $\eta = \eta(\delta, Q)$ as in Lemma~\ref{lem:eta}.

\begin{lemma}\label{lem:liminf}
Let $p \in \partial G - \Lambda H$. Then for $\gamma \in p$, we have that 
$$
\lim_{t \to \infty} d(\gamma(t), H) = \infty.
$$
\end{lemma}

\begin{proof}
Suppose not, so $\liminf_{t \to \infty} d(\gamma(t), H) < \infty$. Then there exists a sequence $t_i \to \infty$, and some $A \geq 0$ such that $d_G(\gamma(t_i), H) < A$ for all $i$. Let $h_i \in H$ be such that $d_G(\gamma(t_i), h_i) < A$.

Recall that we fixed a generating set $Y$ of $H$, and thus a word metric on $H$. Let $\delta_H$ be such that $H$ is $\delta_H$-hyperbolic.
Choose geodesic rays $\mu_i$ in $H$ based at $1$ such that $\mu_i$ passes within $3\delta_H$ of $h_i$. Let $\iota : H \into G$ be the inclusion map, which is a $(\lambda, \varepsilon)$-quasi-isometric embedding. Then $\iota(\mu_i)$ is a $(\lambda,\varepsilon)$-quasi-geodesic ray at 1. Let $\nu_i$ be a geodesic ray in $\Gamma$ with the same endpoints as $\iota(\mu_i)$. Then
$$
d(\nu_i, \gamma(t_i)) < A + 3\lambda \delta_H + \varepsilon + D(\lambda, \varepsilon),
$$
for every $i$. It is now routine to see that $\nu_i(\infty) \to p$ in $\partial G$. But $\Lambda H$ is a closed subset of $\partial G$, so $p \in \Lambda H$. This is a contradiction, so the lemma follows. 
\end{proof}

We now introduce some notation borrowed from \cite{barrett2018computing}. Let $0 \leq r \leq K \leq R$ be constants, where $R$ is possibly infinite.  Denote by $N_R(H)$ the closed $R$-neighbourhood of $H$ in $\Gamma$, and by $N_{r,R}(H)$ the annular region
$$
N_{r,R}(H) = \set{x \in \Gamma}{r \leq d(x,H) \leq R}. 
$$
Let $C_K(H) = \set{x \in \Gamma}{d(x,H) =K}$, and finally let $A_{r,R,K}(H)$ be the union of the components of $N_{r,R}(H)$ which intersect $C_K(H)$. 
The geometric significance of $A_{r,R,K}(H)$ is that we take an annular region about $H$, but discard the components which are ``too close'' to $H$. This is illustrated in Figure~\ref{fig:arrk-fig}. Following Barrett, we will demonstrate a correspondence between the components of $\partial G - \Lambda H$ and the  components of $A_{r,R,K}(H)$, provided $r$, $R$, and $K$ are sufficiently large. 

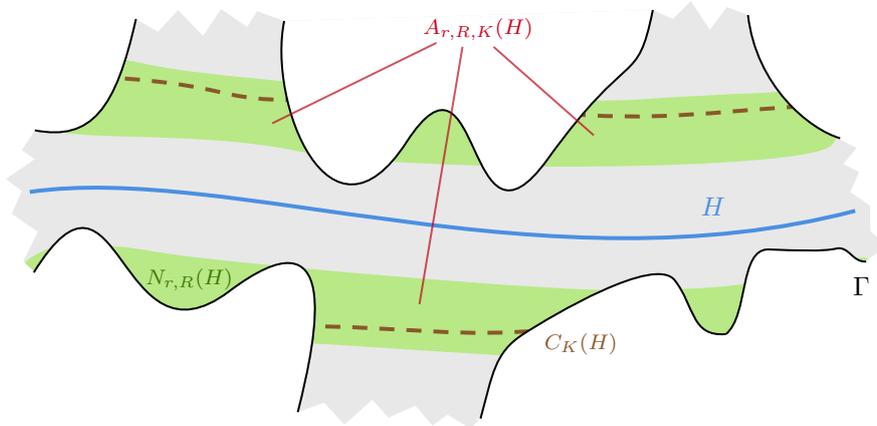
\begin{figure}[b]
    \centering

% Pattern Info
 
\tikzset{
pattern size/.store in=\mcSize, 
pattern size = 5pt,
pattern thickness/.store in=\mcThickness, 
pattern thickness = 0.3pt,
pattern radius/.store in=\mcRadius, 
pattern radius = 1pt}
\makeatletter
\pgfutil@ifundefined{pgf@pattern@name@_c38ios2z6}{
\pgfdeclarepatternformonly[\mcThickness,\mcSize]{_c38ios2z6}
{\pgfqpoint{0pt}{0pt}}
{\pgfpoint{\mcSize+\mcThickness}{\mcSize+\mcThickness}}
{\pgfpoint{\mcSize}{\mcSize}}
{
\pgfsetcolor{\tikz@pattern@color}
\pgfsetlinewidth{\mcThickness}
\pgfpathmoveto{\pgfqpoint{0pt}{0pt}}
\pgfpathlineto{\pgfpoint{\mcSize+\mcThickness}{\mcSize+\mcThickness}}
\pgfusepath{stroke}
}}
\makeatother

% Pattern Info
 
\tikzset{
pattern size/.store in=\mcSize, 
pattern size = 5pt,
pattern thickness/.store in=\mcThickness, 
pattern thickness = 0.3pt,
pattern radius/.store in=\mcRadius, 
pattern radius = 1pt}
\makeatletter
\pgfutil@ifundefined{pgf@pattern@name@_84xy0b260}{
\pgfdeclarepatternformonly[\mcThickness,\mcSize]{_84xy0b260}
{\pgfqpoint{0pt}{0pt}}
{\pgfpoint{\mcSize+\mcThickness}{\mcSize+\mcThickness}}
{\pgfpoint{\mcSize}{\mcSize}}
{
\pgfsetcolor{\tikz@pattern@color}
\pgfsetlinewidth{\mcThickness}
\pgfpathmoveto{\pgfqpoint{0pt}{0pt}}
\pgfpathlineto{\pgfpoint{\mcSize+\mcThickness}{\mcSize+\mcThickness}}
\pgfusepath{stroke}
}}
\makeatother

% Pattern Info
 
\tikzset{
pattern size/.store in=\mcSize, 
pattern size = 5pt,
pattern thickness/.store in=\mcThickness, 
pattern thickness = 0.3pt,
pattern radius/.store in=\mcRadius, 
pattern radius = 1pt}
\makeatletter
\pgfutil@ifundefined{pgf@pattern@name@_mfnqd9z6g}{
\pgfdeclarepatternformonly[\mcThickness,\mcSize]{_mfnqd9z6g}
{\pgfqpoint{0pt}{0pt}}
{\pgfpoint{\mcSize+\mcThickness}{\mcSize+\mcThickness}}
{\pgfpoint{\mcSize}{\mcSize}}
{
\pgfsetcolor{\tikz@pattern@color}
\pgfsetlinewidth{\mcThickness}
\pgfpathmoveto{\pgfqpoint{0pt}{0pt}}
\pgfpathlineto{\pgfpoint{\mcSize+\mcThickness}{\mcSize+\mcThickness}}
\pgfusepath{stroke}
}}
\makeatother
\tikzset{every picture/.style={line width=0.75pt}} %set default line width to 0.75pt        

\begin{tikzpicture}[x=0.75pt,y=0.75pt,yscale=-1,xscale=1]

%Shape: Rectangle [id:dp7791335282385037] 
\draw  [draw opacity=0][fill={rgb, 255:red, 255; green, 255; blue, 255 }  ,fill opacity=1 ] (81,20) -- (527.5,20) -- (527.5,238.5) -- (81,238.5) -- cycle ;
%Shape: Polygon [id:ds15120452749456592] 
\draw  [draw opacity=0][fill={rgb, 255:red, 232; green, 232; blue, 232 }  ,fill opacity=1 ] (221.5,31.5) -- (405.5,26) -- (415.5,30.5) -- (424.5,21.5) -- (434,36) -- (442.5,24.5) -- (453,28.5) -- (499,90.5) -- (507.5,95) -- (500.5,103.5) -- (508.5,107.5) -- (504,117.5) -- (514,124.5) -- (514,137) -- (522.5,145.5) -- (512,152.5) -- (442,189) -- (319.5,231.5) -- (314.5,236.5) -- (303,229.5) -- (295.5,235) -- (287,225) -- (273.5,231) -- (265.5,222) -- (251.5,234.5) -- (241.5,223.5) -- (228.27,228.24) -- (96.5,158) -- (91,153) -- (98.5,137.5) -- (85.5,128.5) -- (90.5,117) -- (82.5,111) -- (97,101) -- (90.5,96.5) -- (97,86.5) -- (147.5,30.5) -- (158.5,37) -- (168.5,26) -- (181,39) -- (190,25.5) -- (203,34.5) -- (208,26.5) -- cycle ;
%Shape: Polygon Curved [id:ds8408974446347506] 
\draw  [color={rgb, 255:red, 184; green, 233; blue, 134 }  ,draw opacity=1 ][fill={rgb, 255:red, 184; green, 233; blue, 134 }  ,fill opacity=1 ] (395.5,167.5) .. controls (435,166.5) and (453.5,166.5) .. (493,155) .. controls (532.5,143.5) and (507,158.5) .. (487,187.5) .. controls (467,216.5) and (271,192.5) .. (196,191.5) .. controls (121,190.5) and (81.5,160) .. (94.5,151.5) .. controls (107.5,143) and (114,144) .. (130.5,145) .. controls (147,146) and (138.44,145.22) .. (166.5,149.5) .. controls (194.56,153.78) and (356,168.5) .. (395.5,167.5) -- cycle ;
%Shape: Polygon Curved [id:ds7513111570211373] 
\draw  [color={rgb, 255:red, 184; green, 233; blue, 134 }  ,draw opacity=1 ][fill={rgb, 255:red, 184; green, 233; blue, 134 }  ,fill opacity=1 ] (398,71.5) .. controls (450,68) and (466.5,68.5) .. (473,67) .. controls (479.5,65.5) and (505.5,80.5) .. (494.5,93.5) .. controls (483.5,106.5) and (283,108) .. (228.5,96.5) .. controls (174,85) and (103.5,93) .. (103.5,83) .. controls (103.5,73) and (130,55.75) .. (132.5,52) .. controls (135,48.25) and (140.44,52.22) .. (168.5,56.5) .. controls (196.56,60.78) and (346,75) .. (398,71.5) -- cycle ;
%Curve Lines [id:da31321314098803876] 
\draw [color={rgb, 255:red, 74; green, 144; blue, 226 }  ,draw opacity=1 ][line width=1.5]    (94.5,117.5) .. controls (194.5,103) and (351,168) .. (506.5,127) ;
%Curve Lines [id:da2469196100755433] 
\draw [color={rgb, 255:red, 139; green, 87; blue, 42 }  ,draw opacity=1 ][line width=1.5]  [dash pattern={on 5.63pt off 4.5pt}]  (128.5,59.5) .. controls (210,64.5) and (178.5,74) .. (254.5,70.5) ;
%Curve Lines [id:da9016152211414479] 
\draw [color={rgb, 255:red, 139; green, 87; blue, 42 }  ,draw opacity=1 ][line width=1.5]  [dash pattern={on 5.63pt off 4.5pt}]  (351,77.5) .. controls (412.5,82.5) and (432.5,77) .. (480,74.5) ;
%Curve Lines [id:da9705893235391396] 
\draw [color={rgb, 255:red, 139; green, 87; blue, 42 }  ,draw opacity=1 ][line width=1.5]  [dash pattern={on 5.63pt off 4.5pt}]  (228.5,185) .. controls (280.5,185) and (310.5,191.5) .. (355.5,186.5) ;
%Shape: Polygon [id:ds6660411297777995] 
\draw  [color={rgb, 255:red, 255; green, 255; blue, 255 }  ,draw opacity=1 ][fill={rgb, 255:red, 255; green, 255; blue, 255 }  ,fill opacity=1 ] (96.5,158) -- (228.27,228.24) -- (208,236.5) -- (86,235.5) -- (91,153) -- cycle ;
%Shape: Polygon [id:ds16008677956892736] 
\draw  [color={rgb, 255:red, 255; green, 255; blue, 255 }  ,draw opacity=1 ][fill={rgb, 255:red, 255; green, 255; blue, 255 }  ,fill opacity=1 ] (512,152.5) -- (522.5,145.5) -- (524.5,234) -- (319.5,231.5) -- (442,189) -- cycle ;
%Shape: Polygon [id:ds16364394142194882] 
\draw  [color={rgb, 255:red, 255; green, 255; blue, 255 }  ,draw opacity=1 ][fill={rgb, 255:red, 255; green, 255; blue, 255 }  ,fill opacity=1 ] (502,29) -- (518.5,35.5) -- (507.5,95) -- (499,90.5) -- (453,28.5) -- cycle ;
%Shape: Polygon [id:ds3271964216460801] 
\draw  [color={rgb, 255:red, 255; green, 255; blue, 255 }  ,draw opacity=1 ][fill={rgb, 255:red, 255; green, 255; blue, 255 }  ,fill opacity=1 ] (147.5,30.5) -- (97,86.5) -- (90,64.5) -- (91.5,49.5) -- (93,29.5) -- cycle ;
%Curve Lines [id:da7999009216250108] 
\draw [fill={rgb, 255:red, 255; green, 255; blue, 255 }  ,fill opacity=1 ]   (512,152.5) .. controls (505,153) and (503.5,144.5) .. (497,146) .. controls (490.5,147.5) and (478.5,147) .. (463.45,146.43) .. controls (448.4,145.86) and (453.5,179.5) .. (442,189) ;
%Shape: Polygon [id:ds009641575039035288] 
\draw  [draw opacity=0][pattern=_c38ios2z6,pattern size=5.699999999999999pt,pattern thickness=0.75pt,pattern radius=0pt, pattern color={rgb, 255:red, 208; green, 2; blue, 27}] (189.5,59) -- (223,63) -- (237,97) -- (177,91) -- (111,88) -- (142,51) -- cycle ;
%Shape: Polygon [id:ds18883414038406143] 
\draw  [draw opacity=0][pattern=_84xy0b260,pattern size=5.699999999999999pt,pattern thickness=0.75pt,pattern radius=0pt, pattern color={rgb, 255:red, 208; green, 2; blue, 27}] (427,70) -- (469,65.5) -- (499,90.5) -- (491.5,97.5) -- (417,104) -- (346,104.5) -- (373.5,71) -- cycle ;
%Shape: Polygon [id:ds16741188409945185] 
\draw  [draw opacity=0][pattern=_mfnqd9z6g,pattern size=5.699999999999999pt,pattern thickness=0.75pt,pattern radius=0pt, pattern color={rgb, 255:red, 208; green, 2; blue, 27}] (300.5,161.5) -- (381,166.5) -- (353,188.5) -- (336.5,200) -- (275.5,196.5) -- (228,193) -- (229.5,156.5) -- cycle ;
%Curve Lines [id:da6156551832438981] 
\draw [fill={rgb, 255:red, 255; green, 255; blue, 255 }  ,fill opacity=1 ]   (97,86.5) .. controls (125.5,92) and (137.5,75.5) .. (147.5,30.5) ;
%Curve Lines [id:da009163868401560027] 
\draw [fill={rgb, 255:red, 255; green, 255; blue, 255 }  ,fill opacity=1 ]   (221.5,31.5) .. controls (213,73) and (244.5,151.5) .. (282.5,92.5) .. controls (320.5,33.5) and (311.5,159) .. (353,101) .. controls (394.5,43) and (397.5,65.5) .. (405.5,26) ;
%Curve Lines [id:da9015026137680677] 
\draw [fill={rgb, 255:red, 255; green, 255; blue, 255 }  ,fill opacity=1 ]   (453,28.5) .. controls (457,64) and (483.5,86) .. (499,90.5) ;
%Curve Lines [id:da8927729781981415] 
\draw [fill={rgb, 255:red, 255; green, 255; blue, 255 }  ,fill opacity=1 ]   (228.27,228.24) .. controls (238.63,187.16) and (250,125) .. (194,167.5) .. controls (138,210) and (140,87.5) .. (96.5,158) ;
%Curve Lines [id:da2248202904177241] 
\draw [fill={rgb, 255:red, 255; green, 255; blue, 255 }  ,fill opacity=1 ]   (442,189) .. controls (424.5,190.5) and (426.5,175) .. (416.5,162) .. controls (406.5,149) and (355.5,179.5) .. (340.95,188.93) .. controls (326.4,198.36) and (325.51,209.25) .. (319.5,231.5) ;
%Straight Lines [id:da13181456984772177] 
\draw [color={rgb, 255:red, 190; green, 0; blue, 25 }  ,draw opacity=0.69 ]   (297.5,42.5) -- (215.5,83) ;
%Straight Lines [id:da3606954841066279] 
\draw [color={rgb, 255:red, 190; green, 0; blue, 25 }  ,draw opacity=0.69 ]   (326,44.5) -- (376,89) ;
%Straight Lines [id:da006074176840624945] 
\draw [color={rgb, 255:red, 190; green, 0; blue, 25 }  ,draw opacity=0.69 ]   (310.5,44.5) -- (289,174) ;

% Text Node
\draw (428,118.9) node [anchor=north west][inner sep=0.75pt]    {$\textcolor[rgb]{0.29,0.56,0.89}{H}$};
% Text Node
\draw (349.95,186.83) node [anchor=north west][inner sep=0.75pt]  [font=\footnotesize,color={rgb, 255:red, 139; green, 87; blue, 42 }  ,opacity=1 ]  {$C_{K}( H)$};
% Text Node
\draw (151,154.4) node [anchor=north west][inner sep=0.75pt]  [font=\footnotesize,color={rgb, 255:red, 65; green, 117; blue, 5 }  ,opacity=1 ]  {$N_{r,R}( H)$};
% Text Node
\draw (504,158.9) node [anchor=north west][inner sep=0.75pt]    {$\Gamma $};
% Text Node
\draw (289.5,28.4) node [anchor=north west][inner sep=0.75pt]  [font=\footnotesize,color={rgb, 255:red, 208; green, 2; blue, 27 }  ,opacity=1 ]  {$A_{r,R,K}( H)$};

\end{tikzpicture}

    \caption{The region $A_{r,R,K}(H)$ surrounding $H$. In this cartoon $H$ is depicted as a line for illustrative purposes.}
    \label{fig:arrk-fig}
\end{figure}

Next, we introduce ``shadows''. Given a component $U$ of $A_{r,\infty,K}(H)$, define its \textit{shadow} $\shad U$ as the set of points $p \in \partial G$ such that for any ray $\gamma \in p$ based at $1$, we have that $\gamma(t)$ is in $U$ for all sufficiently large $t$. We also extend the definition of a shadow to a union of components $V = \bigcup_i U_i$ by defining 
$
\shad V = \bigcup_i \shad U_i
$.

We now prove five lemmas, generalised from \cite{barrett2018computing}. Originally, these results related to bi-infinite geodesics, but careful inspection of their proofs reveals that we only really need the fact that a geodesic path in a hyperbolic space is a quasiconvex set. This allows us to generalise some of the machinery in this paper from geodesics to quasiconvex subgroups. 
Apart from some small changes, the following proofs are practically reproduced verbatim from \cite{barrett2018computing}. 

\begin{lemma}[{\cite[Lem.~1.17]{barrett2018computing}}]\label{lem:shadow-covers}
For $r > \eta$, $K \geq r$, we have that 
$$
\partial G - \Lambda H = \bigcup_U \mathcal S U, 
$$
where $U$ ranges over the connected components of $A_{r, \infty, K}(H)$. Moreover, $\shad U \cap \shad V = \emptyset$ for distinct components $U$, $V$. 
\end{lemma}

\begin{proof}
It is clear from the definition of $\shad U$ that distinct components have disjoint shadows. 
Since $r > \eta$, we have by Lemma~\ref{lem:eta} that $\shad U$ and $\Lambda H$ are disjoint. 
Now let $p \in \partial G - \Lambda H$. We need to check that $p \in \shad U$ for some component $U$ of $A_{r, \infty, K}(H)$. Let $\gamma \in p$, then there is some $t_0$ such that for $t \geq t_0$ we have that $d(\gamma(t), H) \geq r + D$. 

Let $U$ be the component of $A_{r, \infty, K}(H)$ containing $\gamma (t)$ for $t \geq t_0$. We claim that if $\gamma'$ is another ray in $p$, then, $\gamma'(t')$ lies in $U$ for sufficiently large $t'$. 
We know that the Hausdorff distance between $\gamma$ and $\gamma'$ is at most $D$. In particular, for every $t'$ there is some $t$ such that 
$d(\gamma(t), \gamma'(t')) \leq D$, and by the triangle inequality $t$ and $t'$ satisfy $|t-t'| \leq D$. Then if $t' \geq t_0 + D$, we must have that $t \geq t_0$, and so the segment between $\gamma'(t')$ and $\gamma(t)$ lies in $A_{r, \infty, K}(H)$, and thus $\gamma'(t')\in U$. It follows $p \in \shad U$. 
\end{proof}

\begin{lemma}[{\cite[Lem.~1.18]{barrett2018computing}}]\label{lem:shad-nonempty}
Let $r > \eta$ and $K \geq r + Q + \delta + C$. Then for every component $U$ of $A_{r,\infty,K}(H)$, we have that $\mathcal S U$ is non-empty. 
\end{lemma}

\begin{proof}
Let $x \in C_K(H) \cap U$, and using Lemma~\ref{lem:visibility} choose a geodesic ray $\gamma$ based at 1 which passes within $C$ of $x$, say $d(\gamma(t), x) \leq C$. Since $K > r + C$, we have that $\gamma(t) \in U$. In particular, 
$$
d(\gamma(t), H) > r + Q + \delta. 
$$
Suppose now that for some $t' \geq t$, we have that $d(\gamma(t'), H) \leq r$. Let $y \in H$ be a nearest point projection of $\gamma(t')$. Then consider the geodesic triangle $[1,y,\gamma(t')]$, where the segment $[1,\gamma(t')]$ is precisely an initial segment of $\gamma$. Then $\gamma(t)$ is either $\delta$-close to $[y, \gamma(t')]$ or $[1,y]$. In any case, 
$$
d(\gamma(t), H) \leq r +Q + \delta,
$$
which is a contradiction. Thus $\gamma(t') \in U$ for every $t' \geq t$. 
It now follows by an argument similar to that in the proof of Lemma~\ref{lem:shadow-covers} that $\gamma(\infty) \in \shad U$. 
\end{proof}

\begin{lemma}[{\cite[Lem.~1.19]{barrett2018computing}}]
Let $r > \eta$, $K > r$, then for every component $U$ of $A_{r,\infty,K}(H)$, we have that $\mathcal S U$ is open and closed in $\partial G - \Lambda H$.
\end{lemma}

\begin{proof}
Let $\gamma \in p \in \shad U$, and for $t \geq 0$ let 
$$
V_t(\gamma) = \{\beta(\infty) : \text{$\beta$ is a ray based at 1},\  d(\gamma(t), \beta(t)) < 2\delta + 1 \}.
$$
As $t$ varies, this forms a basis of neighbourhoods about $p$. 
By Lemma~\ref{lem:liminf} there exists a $t_0$ such that for all $t \geq t_0$, we have that 
$$
d(\gamma(t), H) > r + Q + 7\delta + 1.
$$
We claim that $V_{t_0}(\gamma) \subseteq \shad U$, which implies that $\shad U$ is open. 

Let $\beta \in q \in V_{t_0}(\gamma)$, so by definition we have $d(\beta(t_0), \gamma(t_0)) < 2\delta + 1$. 
Then consider another ray $\beta' \in q$ also, so $d(\beta(t_0), \beta'(t_0)) < 4\delta$. We need to show that $\beta'(t) \in U$ for all $t \geq t_0$. 
Suppose this is not the case. We have that $\beta'(t_0)$ is in $U$, so for this not to hold we must have that $\beta'$ ``leaves'' $U$ at some point. Thus for some $t' \geq t_0$, we have that $d(\beta'(t'), H) \leq r$, say for some $y \in H$ we have $d(\beta'(t'), y) \leq r$. 
\begin{figure}[t]
    \centering

\tikzset{every picture/.style={line width=0.75pt}} %set default line width to 0.75pt        

\begin{tikzpicture}[x=0.75pt,y=0.75pt,yscale=-1,xscale=1]
%uncomment if require: \path (0,330); %set diagram left start at 0, and has height of 330

%Curve Lines [id:da9639712674011425] 
\draw [fill={rgb, 255:red, 242; green, 242; blue, 242 }  ,fill opacity=1 ]   (182.86,247.03) .. controls (189.24,244.25) and (212.86,204.42) .. (246.44,200.67) .. controls (280.03,196.93) and (316.57,195.96) .. (334.76,204.28) .. controls (352.95,212.61) and (417.1,228.71) .. (437.52,213.17) .. controls (457.95,197.62) and (464.97,219.27) .. (488.58,203.73) .. controls (512.2,188.18) and (456.03,222.6) .. (464.33,233.15) .. controls (472.63,243.7) and (465.61,220.38) .. (482.2,249.25) ;
%Straight Lines [id:da5131123894284668] 
\draw [color={rgb, 255:red, 103; green, 103; blue, 103 }  ,draw opacity=1 ][line width=0.75]  [dash pattern={on 0.84pt off 2.51pt}]  (229.45,216.77) -- (431.78,143.22) ;
%Straight Lines [id:da9382396560167214] 
\draw [color={rgb, 255:red, 155; green, 155; blue, 155 }  ,draw opacity=1 ] [dash pattern={on 4.5pt off 4.5pt}]  (280.51,121.62) -- (368.27,110.79) ;
%Straight Lines [id:da09219476531610904] 
\draw [color={rgb, 255:red, 155; green, 155; blue, 155 }  ,draw opacity=1 ] [dash pattern={on 4.5pt off 4.5pt}]  (280.51,122.12) -- (330.61,179.99) ;
\draw [shift={(330.61,179.99)}, rotate = 49.12] [color={rgb, 255:red, 155; green, 155; blue, 155 }  ,draw opacity=1 ][fill={rgb, 255:red, 155; green, 155; blue, 155 }  ,fill opacity=1 ][line width=0.75]      (0, 0) circle [x radius= 3.35, y radius= 3.35]   ;
%Curve Lines [id:da6881016010188186] 
\draw [fill={rgb, 255:red, 242; green, 242; blue, 242 }  ,fill opacity=1 ]   (486.03,240.37) .. controls (420.93,236.48) and (429.23,236.48) .. (437.52,213.17) .. controls (445.82,189.85) and (434.65,155.98) .. (431.78,143.22) .. controls (428.91,130.45) and (457.95,117.68) .. (470.08,102.13) .. controls (482.2,86.59) and (510.73,72.45) .. (542,78) ;
%Straight Lines [id:da8054528329625135] 
\draw    (94.86,71.13) -- (229.45,216.77) ;
\draw [shift={(229.45,216.77)}, rotate = 47.26] [color={rgb, 255:red, 0; green, 0; blue, 0 }  ][fill={rgb, 255:red, 0; green, 0; blue, 0 }  ][line width=0.75]      (0, 0) circle [x radius= 3.35, y radius= 3.35]   ;
\draw [shift={(93.5,69.66)}, rotate = 47.26] [color={rgb, 255:red, 0; green, 0; blue, 0 }  ][line width=0.75]    (10.93,-3.29) .. controls (6.95,-1.4) and (3.31,-0.3) .. (0,0) .. controls (3.31,0.3) and (6.95,1.4) .. (10.93,3.29)   ;
%Straight Lines [id:da4152720387334803] 
\draw    (200.46,63.57) -- (229.45,216.77) ;
\draw [shift={(229.45,216.77)}, rotate = 79.29] [color={rgb, 255:red, 0; green, 0; blue, 0 }  ][fill={rgb, 255:red, 0; green, 0; blue, 0 }  ][line width=0.75]      (0, 0) circle [x radius= 3.35, y radius= 3.35]   ;
\draw [shift={(200.09,61.61)}, rotate = 79.29] [color={rgb, 255:red, 0; green, 0; blue, 0 }  ][line width=0.75]    (10.93,-3.29) .. controls (6.95,-1.4) and (3.31,-0.3) .. (0,0) .. controls (3.31,0.3) and (6.95,1.4) .. (10.93,3.29)   ;
%Straight Lines [id:da421530943812668] 
\draw    (313.38,62.26) -- (229.45,216.77) ;
\draw [shift={(229.45,216.77)}, rotate = 118.51] [color={rgb, 255:red, 0; green, 0; blue, 0 }  ][fill={rgb, 255:red, 0; green, 0; blue, 0 }  ][line width=0.75]      (0, 0) circle [x radius= 3.35, y radius= 3.35]   ;
\draw [shift={(314.34,60.5)}, rotate = 118.51] [color={rgb, 255:red, 0; green, 0; blue, 0 }  ][line width=0.75]    (10.93,-3.29) .. controls (6.95,-1.4) and (3.31,-0.3) .. (0,0) .. controls (3.31,0.3) and (6.95,1.4) .. (10.93,3.29)   ;
%Straight Lines [id:da2454237410594613] 
\draw [color={rgb, 255:red, 74; green, 144; blue, 226 }  ,draw opacity=1 ]   (147.75,127.39) -- (209.66,114.9) ;
\draw [shift={(209.66,114.9)}, rotate = 348.59] [color={rgb, 255:red, 74; green, 144; blue, 226 }  ,draw opacity=1 ][fill={rgb, 255:red, 74; green, 144; blue, 226 }  ,fill opacity=1 ][line width=0.75]      (0, 0) circle [x radius= 3.35, y radius= 3.35]   ;
\draw [shift={(147.75,127.39)}, rotate = 348.59] [color={rgb, 255:red, 74; green, 144; blue, 226 }  ,draw opacity=1 ][fill={rgb, 255:red, 74; green, 144; blue, 226 }  ,fill opacity=1 ][line width=0.75]      (0, 0) circle [x radius= 3.35, y radius= 3.35]   ;
%Straight Lines [id:da264145251536372] 
\draw [color={rgb, 255:red, 74; green, 144; blue, 226 }  ,draw opacity=1 ]   (209.66,114.9) -- (280.51,122.12) ;
\draw [shift={(280.51,122.12)}, rotate = 5.82] [color={rgb, 255:red, 74; green, 144; blue, 226 }  ,draw opacity=1 ][fill={rgb, 255:red, 74; green, 144; blue, 226 }  ,fill opacity=1 ][line width=0.75]      (0, 0) circle [x radius= 3.35, y radius= 3.35]   ;
\draw [shift={(209.66,114.9)}, rotate = 5.82] [color={rgb, 255:red, 74; green, 144; blue, 226 }  ,draw opacity=1 ][fill={rgb, 255:red, 74; green, 144; blue, 226 }  ,fill opacity=1 ][line width=0.75]      (0, 0) circle [x radius= 3.35, y radius= 3.35]   ;
%Straight Lines [id:da20208409981062458] 
\draw [color={rgb, 255:red, 208; green, 2; blue, 27 }  ,draw opacity=1 ]   (304.77,79.37) -- (431.78,143.22) ;
\draw [shift={(431.78,143.22)}, rotate = 26.69] [color={rgb, 255:red, 208; green, 2; blue, 27 }  ,draw opacity=1 ][fill={rgb, 255:red, 208; green, 2; blue, 27 }  ,fill opacity=1 ][line width=0.75]      (0, 0) circle [x radius= 3.35, y radius= 3.35]   ;
\draw [shift={(304.77,79.37)}, rotate = 26.69] [color={rgb, 255:red, 208; green, 2; blue, 27 }  ,draw opacity=1 ][fill={rgb, 255:red, 208; green, 2; blue, 27 }  ,fill opacity=1 ][line width=0.75]      (0, 0) circle [x radius= 3.35, y radius= 3.35]   ;
%Shape: Parallelogram [id:dp5027273781122095] 
\draw  [draw opacity=0][fill={rgb, 255:red, 242; green, 242; blue, 242 }  ,fill opacity=1 ] (434.35,215.39) -- (494.97,215.39) -- (484.1,249.25) -- (423.48,249.25) -- cycle ;
%Straight Lines [id:da7238389381214603] 
\draw [color={rgb, 255:red, 65; green, 117; blue, 5 }  ,draw opacity=1 ]   (330.61,179.99) -- (334.76,204.28) ;
\draw [shift={(334.76,204.28)}, rotate = 80.31] [color={rgb, 255:red, 65; green, 117; blue, 5 }  ,draw opacity=1 ][fill={rgb, 255:red, 65; green, 117; blue, 5 }  ,fill opacity=1 ][line width=0.75]      (0, 0) circle [x radius= 3.35, y radius= 3.35]   ;
\draw [shift={(330.61,179.99)}, rotate = 80.31] [color={rgb, 255:red, 65; green, 117; blue, 5 }  ,draw opacity=1 ][fill={rgb, 255:red, 65; green, 117; blue, 5 }  ,fill opacity=1 ][line width=0.75]      (0, 0) circle [x radius= 3.35, y radius= 3.35]   ;

% Text Node
\draw (113.24,64.83) node [anchor=north west][inner sep=0.75pt]  [font=\small]  {$\gamma $};
% Text Node
\draw (212.42,61.23) node [anchor=north west][inner sep=0.75pt]  [font=\small]  {$\beta $};
% Text Node
\draw (318.45,64.23) node [anchor=north west][inner sep=0.75pt]  [font=\small]  {$\beta '$};
% Text Node
\draw (2887.2,73.16) node [anchor=north west][inner sep=0.75pt]  [font=\footnotesize,color={rgb, 255:red, 208; green, 2; blue, 27 }  ,opacity=1 ]  {$t'$};
% Text Node
\draw (440.92,137.56) node [anchor=north west][inner sep=0.75pt]  [font=\footnotesize,color={rgb, 255:red, 208; green, 2; blue, 27 }  ,opacity=1 ]  {$y$};
% Text Node
\draw (140.93,135.95) node [anchor=north west][inner sep=0.75pt]  [font=\footnotesize,color={rgb, 255:red, 74; green, 144; blue, 226 }  ,opacity=1 ]  {$t_{0}$};
% Text Node
\draw (216.8,121.52) node [anchor=north west][inner sep=0.75pt]  [font=\footnotesize,color={rgb, 255:red, 74; green, 144; blue, 226 }  ,opacity=1 ]  {$t_{0}$};
% Text Node
\draw (276.34,130.52) node [anchor=north west][inner sep=0.75pt]  [font=\footnotesize,color={rgb, 255:red, 74; green, 144; blue, 226 }  ,opacity=1 ]  {$t_{0}$};
% Text Node
\draw (143,104.53) node [anchor=north west][inner sep=0.75pt]  [font=\footnotesize,color={rgb, 255:red, 74; green, 144; blue, 226 }  ,opacity=1 ]  {$\leq 2\delta +1$};
% Text Node
\draw (227.52,101.75) node [anchor=north west][inner sep=0.75pt]  [font=\footnotesize,color={rgb, 255:red, 74; green, 144; blue, 226 }  ,opacity=1 ]  {$\leq 4\delta $};
% Text Node
\draw (475.8,199.13) node [anchor=north west][inner sep=0.75pt]    {$H$};
% Text Node
\draw (362.85,96.14) node [anchor=north west][inner sep=0.75pt]  [font=\footnotesize,color={rgb, 255:red, 208; green, 2; blue, 27 }  ,opacity=1 ]  {$\leq \ r$};
% Text Node
\draw (221.35,224.11) node [anchor=north west][inner sep=0.75pt]  [font=\footnotesize]  {$1$};
% Text Node
\draw (316.94,120.9) node [anchor=north west][inner sep=0.75pt]  [font=\footnotesize,color={rgb, 255:red, 155; green, 155; blue, 155 }  ,opacity=1 ]  {$\leq \delta ?$};
% Text Node
\draw (281.73,161.9) node [anchor=north west][inner sep=0.75pt]  [font=\footnotesize,color={rgb, 255:red, 155; green, 155; blue, 155 }  ,opacity=1 ]  {$\leq \delta ?$};
% Text Node
\draw (334.7,186.9) node [anchor=north west][inner sep=0.75pt]  [font=\footnotesize,color={rgb, 255:red, 65; green, 117; blue, 5 }  ,opacity=1 ]  {$\leq Q$};

\end{tikzpicture}

    \caption{Bounding the distance between $\gamma(t_0)$ and $H$.}
    \label{fig:bounding-dist-lem3}
\end{figure}
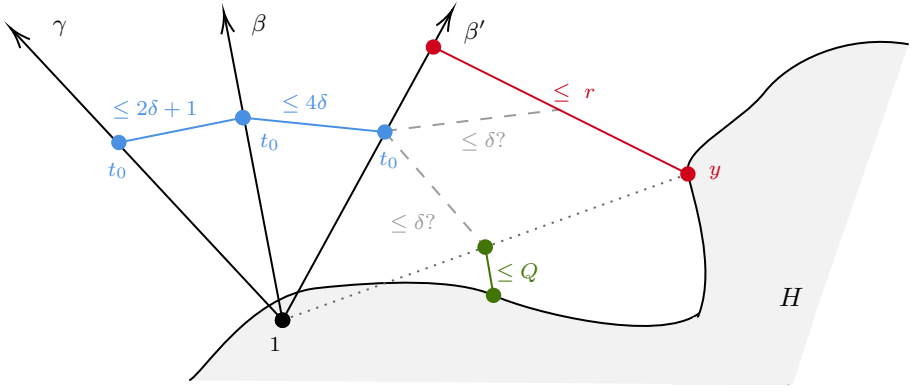
Consider a geodesic path $[1,y]$. Since $H$ is $Q$-quasiconvex this path is contained in the $Q$-neighbourhood of $H$. Combine this with the fact that the triangle $[1, \beta'(t'), y]$ is $\delta$-slim, then inspection of Figure~\ref{fig:bounding-dist-lem3} reveals that 
$$
d(\gamma(t_0),H) \leq r + Q + 7\delta + 1.
$$
This is a contradiction, so $\shad U$ is open. 

To conclude, note that the shadows $\shad U$ form a disjoint open cover of $\partial G - \Lambda H$. It follows that that they must also be closed. 
\end{proof}

For the next lemma, recall that $a_1$, $a_2$ and $k$ denote the parameters of our fixed visual metric on $\partial G$. Also recall that $n$ has been fixed so that $\ddagger_n$ holds in $\Gamma$. 

\begin{lemma}[{\cite[Lem.~1.20]{barrett2018computing}}]\label{lem:projecting-paths}
Suppose that $r$ satisfies
$$
r > 2 \log_a\left( \frac {k_2(n-1)}{k_1(1-a^{-1})}\right) + M + 8\delta + \eta  + C,
$$ 
and let $K > r$. Then for every component $U$ of $A_{r,\infty, K}(H)$, $\mathcal S U$ is contained within exactly one connected component of $\partial G - \Lambda H$. 
Moreover, every component of $\partial G - \Lambda H$ is path connected. 
\end{lemma}

\begin{proof}
Let $p, q \in \shad U$, and let $\alpha_1 \in p$, $\alpha_2 \in q$. Then there is some $t_1$, $t_2$ such that $\alpha_1(t_1), \alpha_2(t_2) \in U$. Let $\phi : [0,\ell] \to U$ be a path connecting  $\alpha_1(t_1)$ to $\alpha_2(t_2)$, parameterised by arc length. 

We follow the methodology of Bestvina--Mess \cite{bestvina1991boundary} and ``project'' this path to the boundary. For every $i \in \Z \cap [0,\ell]$, let $\beta_i$ be a ray from $1$, passing through a point $\beta_i(m_i) = z_i$ such that $d(z_i, \phi(i)) \leq C$. 
We will show that $\beta_i(\infty)$ and $\beta_{i+1}(\infty)$ can be connected by a path in $\partial G$ avoiding $\Lambda H$. This implies the result. Note that each $\beta_i(\infty)$ is not in $\Lambda H$, since $r \geq \eta + C$. 

For notational convenience, assume that $i = 0$. Let $n$ be such that $\ddagger_n$ holds in $\Gamma$. For every $n$-adic $t \in [0,1]$, by induction on the power $k$ of the denominator we define $\beta_t$, satisfying
$$
d(\beta_{j/n^k}(m_i + k),\beta_{j+1/n^k}(m_i + k)) \leq M,
$$
where $M = 6C + 2\delta + 3$, for each $0 \leq j < n^k$. Note that the base case of this induction holds because $M \geq 2C + 1$. Secondly, 
the triangle inequality then gives the following lower bound:
\begin{align*}
    (\beta_{j/n_k}(\infty) \cdot \beta_{j+1/n_k}(\infty))_1 
    &\geq \liminf_{n_1, n_2} (\beta_{j/n_k}(n_1) \cdot \beta_{j+1/n_k}(n_2))_1 \\
    &\geq (\beta_{j/n_k}(m_0 + k) \cdot \beta_{j+1/n_k}(m_0 + k))_1 \\
    &= m_0 + k - M/2.
\end{align*}
We therefore deduce that 
$$
\rho (\beta_{j/n_k}(\infty), \beta_{j+1/n_k}(\infty)) \leq k_2 a^{-m_0 - k + M/2}.
$$
Inductively applying the triangle inequality and summing the geometric series, we thus obtain the bound that for every $n$-adic rational $t\in [0,1]$, we have 
$$
\rho (\beta_{0}(\infty), \beta_{t}(\infty)) \leq \frac {k_2 (n-1)a^{-m_0  + M/2}}{1-a^{-1}}.
$$
We then define a path $\psi : [0,1] \to \partial G$ by $\psi(t) = \beta_t(\infty)$, for every $n$-adic $t$, and extending to $[0,1]$ continuously. 
We have shown that this path is contained within a ball of radius 
\begin{equation}\label{eq:ball-around-beta0}
  \frac{k_2(n-1) a^{-m_0 + \frac M 2}}{1-a^{-1}}  
\end{equation}
around $\beta_0(\infty)$. 

We now seek to place a lower bound on $\rho (\beta_0(\infty), \Lambda H)$ in terms of $m_0$. 
Let $\gamma$ be a ray from 1 with limit in $\Lambda H$. So $\gamma$ is contained in the $\eta$-neighbourhood of $H$. 
By \cite[ch.~III.H,~3.17]{bridson2013metric} we have that 
\begin{equation}\label{eq:bounding-gromov-prod}
(\beta_0(\infty)\cdot \gamma(\infty))_1 \leq \liminf_{n_1,n_2} (\beta_0(n_1) \cdot \gamma(n_2))+2\delta.
\end{equation}
So let $n_1, n_2 \geq m_0$. We will show that 
$(\beta_0(n_1) \cdot \gamma(n_2))_1$ can be approximated by $(\beta_0(m_0) \cdot \gamma(m_0))_1$.
Since $d(\beta_0(m_0), H) \geq r > \delta + \eta$, we have that $d(\beta_0(m_0) , \gamma) > \delta$. So there exists some $a \in [\beta_0(n_1), \gamma(n_2)]$ such that $d(\beta_0(m_0), a) \leq \delta$. 
Similarly, we have that $d(\gamma(m_0), \beta_0) > \delta$, as $r > 2 \delta + \eta$. Thus again there exists a point $b \in [\beta_0(n_1), \gamma(n_2)]$ such that $d(\gamma(m_0), b) \leq \delta$. 

We now check that $a$ and $b$ are ``in order''. Suppose that $b$ is closer to $\beta_0(n_1)$ than $a$. Then by considering the geodesic triangle $[a, \beta_0(m_0), \beta_0(n_1)]$, we see that $d(b, \beta_0) \leq 2\delta$. But this implies that $d(\beta_0(m_0), \gamma(m_0)) \leq 6\delta$, which is a contradiction since $r > 6\delta + \eta$. Therefore the picture we have looks something like Figure~\ref{fig:lem4}.
\begin{figure}[t]
    \centering

\tikzset{every picture/.style={line width=0.75pt}} %set default line width to 0.75pt        

\begin{tikzpicture}[x=0.75pt,y=0.75pt,yscale=-1,xscale=1]

%Curve Lines [id:da8137421449237892] 
\draw    (100,116.98) .. controls (296.7,116.59) and (409.21,72.28) .. (487.16,31.61) ;
\draw [shift={(488.33,31)}, rotate = 152.37] [color={rgb, 255:red, 0; green, 0; blue, 0 }  ][line width=0.75]    (10.93,-3.29) .. controls (6.95,-1.4) and (3.31,-0.3) .. (0,0) .. controls (3.31,0.3) and (6.95,1.4) .. (10.93,3.29)   ;
\draw [shift={(100,116.98)}, rotate = 359.89] [color={rgb, 255:red, 0; green, 0; blue, 0 }  ][fill={rgb, 255:red, 0; green, 0; blue, 0 }  ][line width=0.75]      (0, 0) circle [x radius= 3.35, y radius= 3.35]   ;
%Curve Lines [id:da8732219796308016] 
\draw    (100,116.98) .. controls (296.7,116.59) and (396.34,145.36) .. (486.97,171.94) ;
\draw [shift={(488.33,172.34)}, rotate = 196.35] [color={rgb, 255:red, 0; green, 0; blue, 0 }  ][line width=0.75]    (10.93,-3.29) .. controls (6.95,-1.4) and (3.31,-0.3) .. (0,0) .. controls (3.31,0.3) and (6.95,1.4) .. (10.93,3.29)   ;
%Straight Lines [id:da5730363417307875] 
\draw [color={rgb, 255:red, 208; green, 2; blue, 27 }  ,draw opacity=1 ]   (425,59.86) -- (440,158.21) ;
\draw [shift={(440,158.21)}, rotate = 81.33] [color={rgb, 255:red, 208; green, 2; blue, 27 }  ,draw opacity=1 ][fill={rgb, 255:red, 208; green, 2; blue, 27 }  ,fill opacity=1 ][line width=0.75]      (0, 0) circle [x radius= 3.35, y radius= 3.35]   ;
\draw [shift={(425,59.86)}, rotate = 81.33] [color={rgb, 255:red, 208; green, 2; blue, 27 }  ,draw opacity=1 ][fill={rgb, 255:red, 208; green, 2; blue, 27 }  ,fill opacity=1 ][line width=0.75]      (0, 0) circle [x radius= 3.35, y radius= 3.35]   ;
%Straight Lines [id:da11112800139965517] 
\draw [color={rgb, 255:red, 144; green, 19; blue, 254 }  ,draw opacity=1 ] [dash pattern={on 0.84pt off 2.51pt}]  (340.5,88.13) -- (430,96.96) ;
\draw [shift={(430,96.96)}, rotate = 5.64] [color={rgb, 255:red, 144; green, 19; blue, 254 }  ,draw opacity=1 ][fill={rgb, 255:red, 144; green, 19; blue, 254 }  ,fill opacity=1 ][line width=0.75]      (0, 0) circle [x radius= 2.34, y radius= 2.34]   ;
\draw [shift={(340.5,88.13)}, rotate = 5.64] [color={rgb, 255:red, 144; green, 19; blue, 254 }  ,draw opacity=1 ][fill={rgb, 255:red, 144; green, 19; blue, 254 }  ,fill opacity=1 ][line width=0.75]      (0, 0) circle [x radius= 2.34, y radius= 2.34]   ;
%Straight Lines [id:da15990706252275722] 
\draw [color={rgb, 255:red, 144; green, 19; blue, 254 }  ,draw opacity=1 ] [dash pattern={on 0.84pt off 2.51pt}]  (341,135.83) -- (435.5,129) ;
\draw [shift={(435.5,129)}, rotate = 355.87] [color={rgb, 255:red, 144; green, 19; blue, 254 }  ,draw opacity=1 ][fill={rgb, 255:red, 144; green, 19; blue, 254 }  ,fill opacity=1 ][line width=0.75]      (0, 0) circle [x radius= 2.34, y radius= 2.34]   ;
\draw [shift={(341,135.83)}, rotate = 355.87] [color={rgb, 255:red, 144; green, 19; blue, 254 }  ,draw opacity=1 ][fill={rgb, 255:red, 144; green, 19; blue, 254 }  ,fill opacity=1 ][line width=0.75]      (0, 0) circle [x radius= 2.34, y radius= 2.34]   ;

% Text Node
\draw (490.5,142.28) node [anchor=north west][inner sep=0.75pt]    {$\gamma $};
% Text Node
\draw (485.5,46.46) node [anchor=north west][inner sep=0.75pt]    {$\beta _{0}$};
% Text Node
\draw (377.5,40.8) node [anchor=north west][inner sep=0.75pt]  [font=\footnotesize,color={rgb, 255:red, 208; green, 2; blue, 27 }  ,opacity=1 ]  {$\beta _{0}( n_{1})$};
% Text Node
\draw (402,163.89) node [anchor=north west][inner sep=0.75pt]  [font=\footnotesize,color={rgb, 255:red, 208; green, 2; blue, 27 }  ,opacity=1 ]  {$\gamma ( n_{2})$};
% Text Node
\draw (84,111.3) node [anchor=north west][inner sep=0.75pt]  [font=\footnotesize]  {$1$};
% Text Node
\draw (385,78.05) node [anchor=north west][inner sep=0.75pt]  [font=\footnotesize,color={rgb, 255:red, 144; green, 19; blue, 254 }  ,opacity=1 ]  {$\leq \delta $};
% Text Node
\draw (376.5,115.05) node [anchor=north west][inner sep=0.75pt]  [font=\footnotesize,color={rgb, 255:red, 144; green, 19; blue, 254 }  ,opacity=1 ]  {$\leq \delta $};
% Text Node
\draw (435.5,86.4) node [anchor=north west][inner sep=0.75pt]  [font=\small,color={rgb, 255:red, 144; green, 19; blue, 254 }  ,opacity=1 ]  {$a$};
% Text Node
\draw (441,115.9) node [anchor=north west][inner sep=0.75pt]  [font=\small,color={rgb, 255:red, 144; green, 19; blue, 254 }  ,opacity=1 ]  {$b$};
% Text Node
\draw (302,69.8) node [anchor=north west][inner sep=0.75pt]  [font=\footnotesize,color={rgb, 255:red, 144; green, 19; blue, 254 }  ,opacity=1 ]  {$\beta _{0}( m_{0})$};
% Text Node
\draw (310,139.89) node [anchor=north west][inner sep=0.75pt]  [font=\footnotesize,color={rgb, 255:red, 144; green, 19; blue, 254 }  ,opacity=1 ]  {$\gamma ( m_{0})$};

\end{tikzpicture}

    \caption{The point $a$ is closer to $\beta_0(n_1)$ than $b$.}
    \label{fig:lem4}
\end{figure}
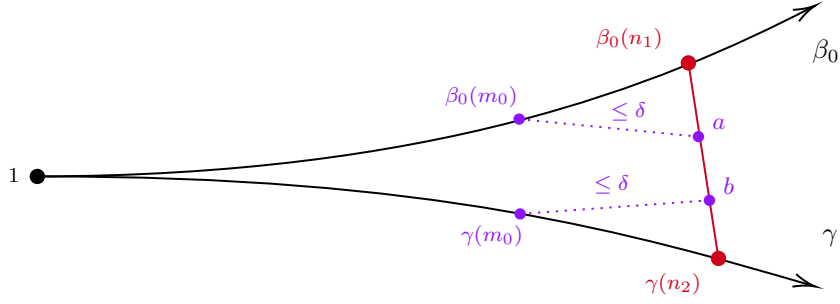

We can thus conclude that 
$$
d(\beta_0(n_1), \gamma(n_2)) = d(\beta_0(n_1), a) + d(a,b) + d(b, \gamma'(n_2)).
$$
We can then compute the following inequality.
\begin{align*}
    (\beta_0(n_1) \cdot \gamma(n_2))_1 - (\beta_0(m_0)) \cdot \gamma(m_0))_1 &= \frac 1 2 [n_1 + n_2 - d(\beta_0(n_1), \gamma(n_2))] \\ 
     &\hspace{1em }- \frac 1 2 [2m_0 - d(\beta_0(m_0), \gamma(m_0))] \\
     &= \frac 1 2 [(n_1 - m_0) + (n_2 - m_0) \\
     & \hspace{1em} + d(\beta_0(m_0), \gamma(m_0)) - d(\beta_0(n_1), \gamma(n_2))] \\
     &= \frac 1 2 [ d(\beta_0(n_1), \beta_0(m_0)) - d(\beta_0(n_1), a) \\
     &\hspace{1.1em} + d(\gamma(n_2), \gamma(m_0)) - d(\gamma(n_2), b) \\
     &\hspace{1.1em} + d(\beta_0(m_0), \gamma(m_0)) - d(a,b)] \\
     &\leq \frac 1 2 [\delta + \delta + 2\delta] = 2\delta. 
\end{align*}
Applying the bound (\ref{eq:bounding-gromov-prod}), it follows that 
\begin{align*}
    (\beta_0(\infty)\cdot \gamma(\infty))_1 &\leq (\beta_0(m_0)\cdot \gamma(m_0))_1 + 4\delta \\
    &= \frac 1 2 [2m_0 - d(\beta_0(m_0), \gamma(m_0))] + 4\delta \\
    &\leq m_0 - \frac 1 2 (r-\eta) + 4\delta.
\end{align*}
This then implies the following lower bound on the distance between $\gamma(\infty)$ and $\beta_0(\infty)$: 
$$
\rho (\beta_0(\infty), \gamma(\infty)) \geq k_1a^{-m_0 - 4\delta + \frac{r-\eta}{2}}.
$$
We combine this inequality with (\ref{eq:ball-around-beta0}), and a simple calculation reveals that our choice of $r$ ensures that the path constructed between $p$ and $q$ avoids $\Lambda H$. The result follows. 
\end{proof}

\begin{lemma}[{\cite[Lem.~1.21]{barrett2018computing}}]
If $R > 4\delta + Q + \max\{r + 4\delta + 1, K\}$, then the inclusion 
$$
A_{r,R,K}(H) \into A_{r,\infty,K}(H) 
$$
induces a bijection between connected components. 
\end{lemma}

\begin{proof}
Surjectivity is obvious since $R \geq K$, so we need only show injectivity. Let $x, y \in C_K(H)$ lie in the same component of $A_{r,\infty,K}(H)$. We claim that the shortest path between $x$ and $y$ in $N_{r, \infty}(H)$ actually lies inside $N_{r,R}(H)$, which implies the result. 

Let $\phi : [0,\ell] \to \Gamma$ be the shortest such path, parameterised by length. Then suppose that for some $s \in [0,\ell]$, we have that $d(\phi(s), H) > R$. 
Let $[t_0, t_1] \subset [0,\ell]$ be the maximal subinterval containing $s$, such that for all $t \in [t_0, t_1]$, we have 
$$
d(\phi(t), H) \geq r+4\delta + 1.
$$
We claim that $\phi|_{[t_0, t_1]}$ is an $(8\delta +2)$-local geodesic.
Indeed, for any $t \in [t_0, t_1]$, we have that the segment 
$$
\phi|_{[t-4\delta -1, t+4\delta+1]\cap [t_0, t_1]}
$$
is contained in $N_{r+4\delta +1, \infty}(H)$.
Therefore, any segment from $\phi(\max\{t_0, t-4\delta-1\})$ to $\phi(\min\{t_1, t+4\delta+1\})$ lies in $N_{r,\infty}(H)$. 
By the minimality of $\phi$, it follows that $\phi|_{[t_0, t_1]}$ is an $(8\delta +2)$-local geodesic.
Apply \cite[ch.~III.H,~1.13]{bridson2013metric}, we then have that any geodesic path between $\phi(t_0)$ and $\phi(t_1)$ is contained in the $2\delta$-neighbourhood of $\phi|_{[t_0, t_1]}$.
Maximality of $[t_0, t_1]$ means that either $t_0 = 0$, or
$
d(\phi(t_0), H) = r+4\delta + 1.
$
In both cases, it follows that $d(\phi(t_0), H) \leq \max \{r+4\delta + 1, K\}$. Identical reasoning also shows that $d(\phi(t_1), H) \leq \max \{r+4\delta + 1, K\}$.

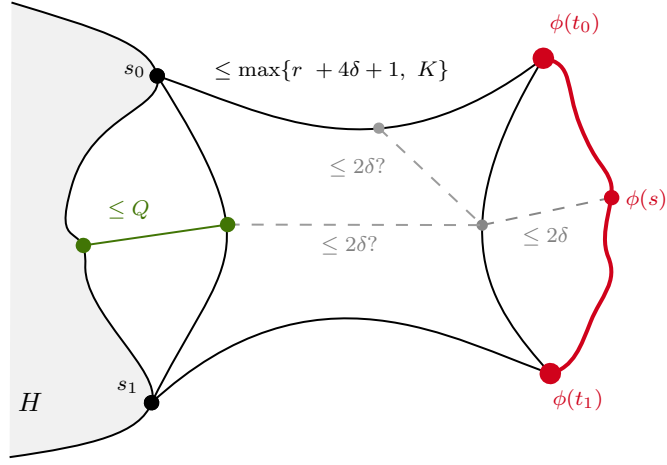
\begin{figure}[h]
    \centering

\tikzset{every picture/.style={line width=0.75pt}} %set default line width to 0.75pt        

\begin{tikzpicture}[x=0.75pt,y=0.75pt,yscale=-1,xscale=1]

%Curve Lines [id:da14412882912858493] 
\draw    (360.1,83.35) .. controls (307,171.5) and (333,207.5) .. (363.57,241.94) ;
\draw [shift={(363.57,241.94)}, rotate = 48.41] [color={rgb, 255:red, 0; green, 0; blue, 0 }  ][fill={rgb, 255:red, 0; green, 0; blue, 0 }  ][line width=0.75]      (0, 0) circle [x radius= 3.35, y radius= 3.35]   ;
\draw [shift={(360.1,83.35)}, rotate = 121.07] [color={rgb, 255:red, 0; green, 0; blue, 0 }  ][fill={rgb, 255:red, 0; green, 0; blue, 0 }  ][line width=0.75]      (0, 0) circle [x radius= 3.35, y radius= 3.35]   ;
%Straight Lines [id:da261578187213638] 
\draw [color={rgb, 255:red, 135; green, 135; blue, 135 }  ,draw opacity=1 ] [dash pattern={on 4.5pt off 4.5pt}]  (329.57,167.27) -- (394.2,153.45) ;
\draw [shift={(394.2,153.45)}, rotate = 347.93] [color={rgb, 255:red, 135; green, 135; blue, 135 }  ,draw opacity=1 ][fill={rgb, 255:red, 135; green, 135; blue, 135 }  ,fill opacity=1 ][line width=0.75]      (0, 0) circle [x radius= 2.34, y radius= 2.34]   ;
\draw [shift={(329.57,167.27)}, rotate = 347.93] [color={rgb, 255:red, 135; green, 135; blue, 135 }  ,draw opacity=1 ][fill={rgb, 255:red, 135; green, 135; blue, 135 }  ,fill opacity=1 ][line width=0.75]      (0, 0) circle [x radius= 2.34, y radius= 2.34]   ;
%Shape: Boxed Bezier Curve [id:dp48411300448002725] 
\draw [fill={rgb, 255:red, 241; green, 241; blue, 241 }  ,fill opacity=1 ]   (94.61,55.31) .. controls (137.13,68.42) and (163.59,75.56) .. (167.37,92.24) .. controls (171.15,108.91) and (135.24,111.29) .. (124.85,147.03) .. controls (114.45,182.76) and (132.41,163.7) .. (131.46,187.52) .. controls (130.52,211.34) and (170.2,241.12) .. (164.53,256.61) .. controls (158.86,272.09) and (135.24,278.04) .. (93.67,284) ;
%Curve Lines [id:da7724783553601875] 
\draw    (167.37,92.24) .. controls (225.5,184.5) and (201,182.5) .. (164.53,256.61) ;
\draw [shift={(164.53,256.61)}, rotate = 116.2] [color={rgb, 255:red, 0; green, 0; blue, 0 }  ][fill={rgb, 255:red, 0; green, 0; blue, 0 }  ][line width=0.75]      (0, 0) circle [x radius= 3.35, y radius= 3.35]   ;
\draw [shift={(167.37,92.24)}, rotate = 57.79] [color={rgb, 255:red, 0; green, 0; blue, 0 }  ][fill={rgb, 255:red, 0; green, 0; blue, 0 }  ][line width=0.75]      (0, 0) circle [x radius= 3.35, y radius= 3.35]   ;
%Curve Lines [id:da5059970410051973] 
\draw    (360.1,83.35) .. controls (278.61,146.44) and (221.97,110.51) .. (167.37,92.24) ;
\draw [shift={(167.37,92.24)}, rotate = 198.51] [color={rgb, 255:red, 0; green, 0; blue, 0 }  ][fill={rgb, 255:red, 0; green, 0; blue, 0 }  ][line width=0.75]      (0, 0) circle [x radius= 3.35, y radius= 3.35]   ;
\draw [shift={(360.1,83.35)}, rotate = 142.26] [color={rgb, 255:red, 0; green, 0; blue, 0 }  ][fill={rgb, 255:red, 0; green, 0; blue, 0 }  ][line width=0.75]      (0, 0) circle [x radius= 3.35, y radius= 3.35]   ;
%Curve Lines [id:da011703835309246235] 
\draw    (363.57,241.94) .. controls (316.12,225.65) and (245.67,181.49) .. (164.53,256.61) ;
\draw [shift={(164.53,256.61)}, rotate = 137.2] [color={rgb, 255:red, 0; green, 0; blue, 0 }  ][fill={rgb, 255:red, 0; green, 0; blue, 0 }  ][line width=0.75]      (0, 0) circle [x radius= 3.35, y radius= 3.35]   ;
\draw [shift={(363.57,241.94)}, rotate = 198.95] [color={rgb, 255:red, 0; green, 0; blue, 0 }  ][fill={rgb, 255:red, 0; green, 0; blue, 0 }  ][line width=0.75]      (0, 0) circle [x radius= 3.35, y radius= 3.35]   ;
%Curve Lines [id:da9999692421677497] 
\draw [color={rgb, 255:red, 208; green, 2; blue, 27 }  ,draw opacity=1 ][line width=1.5]    (360.1,83.35) .. controls (375.71,85.98) and (370.51,97.37) .. (378.6,117.52) .. controls (386.69,137.68) and (390.74,136.8) .. (393.63,145.56) .. controls (396.51,154.32) and (387.27,170.53) .. (392.47,183.24) .. controls (397.67,195.94) and (388.71,203.34) .. (383.8,215.66) .. controls (378.89,227.98) and (375.42,238.88) .. (363.57,241.94) ;
\draw [shift={(363.57,241.94)}, rotate = 165.49] [color={rgb, 255:red, 208; green, 2; blue, 27 }  ,draw opacity=1 ][fill={rgb, 255:red, 208; green, 2; blue, 27 }  ,fill opacity=1 ][line width=1.5]      (0, 0) circle [x radius= 4.36, y radius= 4.36]   ;
\draw [shift={(360.1,83.35)}, rotate = 9.56] [color={rgb, 255:red, 208; green, 2; blue, 27 }  ,draw opacity=1 ][fill={rgb, 255:red, 208; green, 2; blue, 27 }  ,fill opacity=1 ][line width=1.5]      (0, 0) circle [x radius= 4.36, y radius= 4.36]   ;
%Straight Lines [id:da6029312371049658] 
\draw [color={rgb, 255:red, 208; green, 2; blue, 27 }  ,draw opacity=1 ]   (394.2,153.45) ;
\draw [shift={(394.2,153.45)}, rotate = 0] [color={rgb, 255:red, 208; green, 2; blue, 27 }  ,draw opacity=1 ][fill={rgb, 255:red, 208; green, 2; blue, 27 }  ,fill opacity=1 ][line width=0.75]      (0, 0) circle [x radius= 3.35, y radius= 3.35]   ;
%Straight Lines [id:da6377979104411642] 
\draw [color={rgb, 255:red, 155; green, 155; blue, 155 }  ,draw opacity=1 ] [dash pattern={on 4.5pt off 4.5pt}]  (202.5,167) -- (329.57,167.27) ;
\draw [shift={(202.5,167)}, rotate = 0.12] [color={rgb, 255:red, 155; green, 155; blue, 155 }  ,draw opacity=1 ][fill={rgb, 255:red, 155; green, 155; blue, 155 }  ,fill opacity=1 ][line width=0.75]      (0, 0) circle [x radius= 2.34, y radius= 2.34]   ;
%Straight Lines [id:da061433379307024305] 
\draw [color={rgb, 255:red, 155; green, 155; blue, 155 }  ,draw opacity=1 ] [dash pattern={on 4.5pt off 4.5pt}]  (278,118.5) -- (329.57,167.27) ;
\draw [shift={(278,118.5)}, rotate = 43.4] [color={rgb, 255:red, 155; green, 155; blue, 155 }  ,draw opacity=1 ][fill={rgb, 255:red, 155; green, 155; blue, 155 }  ,fill opacity=1 ][line width=0.75]      (0, 0) circle [x radius= 2.34, y radius= 2.34]   ;
%Straight Lines [id:da12929130357855567] 
\draw [color={rgb, 255:red, 65; green, 117; blue, 5 }  ,draw opacity=1 ]   (130.5,177.5) -- (202.5,167) ;
\draw [shift={(202.5,167)}, rotate = 351.7] [color={rgb, 255:red, 65; green, 117; blue, 5 }  ,draw opacity=1 ][fill={rgb, 255:red, 65; green, 117; blue, 5 }  ,fill opacity=1 ][line width=0.75]      (0, 0) circle [x radius= 3.35, y radius= 3.35]   ;
\draw [shift={(130.5,177.5)}, rotate = 351.7] [color={rgb, 255:red, 65; green, 117; blue, 5 }  ,draw opacity=1 ][fill={rgb, 255:red, 65; green, 117; blue, 5 }  ,fill opacity=1 ][line width=0.75]      (0, 0) circle [x radius= 3.35, y radius= 3.35]   ;

% Text Node
\draw (96.49,250.4) node [anchor=north west][inner sep=0.75pt]    {$H$};
% Text Node
\draw (361.86,59.07) node [anchor=north west][inner sep=0.75pt]  [font=\footnotesize,color={rgb, 255:red, 208; green, 2; blue, 27 }  ,opacity=1 ]  {$\phi ( t_{0})$};
% Text Node
\draw (363.31,248.33) node [anchor=north west][inner sep=0.75pt]  [font=\footnotesize,color={rgb, 255:red, 208; green, 2; blue, 27 }  ,opacity=1 ]  {$\phi ( t_{1})$};
% Text Node
\draw (399.73,147.75) node [anchor=north west][inner sep=0.75pt]  [font=\footnotesize,color={rgb, 255:red, 208; green, 2; blue, 27 }  ,opacity=1 ]  {$\phi ( s)$};
% Text Node
\draw (194.36,82.73) node [anchor=north west][inner sep=0.75pt]  [font=\footnotesize]  {$\leq \max\{r\ +4\delta +1,\ K\}$};
% Text Node
\draw (149.22,85.36) node [anchor=north west][inner sep=0.75pt]  [font=\footnotesize]  {$s_{0}$};
% Text Node
\draw (145.75,243.95) node [anchor=north west][inner sep=0.75pt]  [font=\footnotesize]  {$s_{1}$};
% Text Node
\draw (348.14,166.64) node [anchor=north west][inner sep=0.75pt]  [font=\footnotesize,color={rgb, 255:red, 128; green, 128; blue, 128 }  ,opacity=1 ]  {$\leq 2\delta $};
% Text Node
\draw (247.64,171.64) node [anchor=north west][inner sep=0.75pt]  [font=\footnotesize,color={rgb, 255:red, 128; green, 128; blue, 128 }  ,opacity=1 ]  {$\leq 2\delta ?$};
% Text Node
\draw (251.64,132.14) node [anchor=north west][inner sep=0.75pt]  [font=\footnotesize,color={rgb, 255:red, 128; green, 128; blue, 128 }  ,opacity=1 ]  {$\leq 2\delta ?$};
% Text Node
\draw (141.14,153.64) node [anchor=north west][inner sep=0.75pt]  [font=\footnotesize,color={rgb, 255:red, 65; green, 117; blue, 5 }  ,opacity=1 ]  {$\leq Q$};

\end{tikzpicture}

    \caption{The point $\phi(s)$ cannot be too far from $H$.}
    \label{fig:lem-121-square}
\end{figure}

Let $s_0$, $s_1$ be closest-point projections of $\phi(t_0)$, $\phi(t_1)$ respectively. Now, consider the quadrilateral $[s_0, s_1, \phi(t_1), \phi(t_0)]$. 
Any quadrilateral in $\Gamma$ is $2\delta$-thin, so by inspecting Figure~\ref{fig:lem-121-square} we see that in any case, there is a path from $\phi(s)$ to $H$ of length at most $4\delta + Q + \max \{r+4\delta + 1 , K\} = R$. This contradicts our choice of $s$, and so the result follows. 
\end{proof}

In summary, the above lemmas give us the following. 

\begin{proposition}\label{prop:arrk-comp-corr}
Let $G$ be a one-ended hyperbolic group with $\delta$-hyperbolic Cayley graph $\Gamma$, and let $H$ be a $Q$-quasiconvex subgroup. Then there exist computable values $r \leq K \leq R$, such that the map
$$
\pi_0(A_{r,R,K}(H)) \rightarrow \pi_0(A_{r,\infty,K}(H)) \xrightarrow{\shad} \pi_0(\partial G - \Lambda H)
$$
is a well-defined bijection. 
\end{proposition}

For the remainder of this paper we fix $r$, $R$, $K$ such that the above is satisfied. 
We conclude this section with a final remark that $H$ acts on $\pi_0(A_{r,R,K}(H))$ and $\pi_0(\partial G - \Lambda H)$ by permutations, and the above bijection is easily seen to be equivariant with respect to this action.

\subsection{Ends of pairs, revisited}\label{sec:ends-revisited}

We will conclude this section by using the above machinery to relate the number of (filtered) ends of a quasiconvex subgroup to its limit set complement. 
We begin with the following characterisation of filtered ends. 

\begin{theorem}\label{thm:hyp-filteredends}
Let $G$ be a one-ended hyperbolic group and $H$ a quasiconvex subgroup. Then $\tilde e(G,H)$ is equal to the number of components of $\partial G - \Lambda H$. 
\end{theorem}

\begin{proof}
Let $N \in \mathbb N \cup \{\infty\}$ denote the number of components of $\partial G - \Lambda H$. Choose a ray $\gamma_i$ based in each component $C_i$ of $\partial G - \Lambda H$. If $i \neq j$, it is easy to see that there is no filtered homotopy between $\gamma_i$ and $\gamma_j$. This proves that $\tilde e(G,H) \geq N$. 

Conversely, fix a component $C$ of $\partial G - \Lambda H$. It is clear if $\gamma$ is a ray such that $\gamma(\infty) \in C$, then $\gamma$ is a filtered ray.  Now, if $\gamma'$ is another ray such that $\gamma'(\infty) \in C$, then by Lemma~\ref{lem:projecting-paths} we have that $C$ is path connected. So consider a path through $C$ between $\gamma(\infty)$ and $\gamma'(\infty)$, then it is easy to see that this induces a filtered homotopy between $\gamma$ and $\gamma'$ in the Cayley complex. Thus, $\tilde e (G,H) \leq N$. 
\end{proof}

Note that $H$ acts on $\partial G - \Lambda H$ by homeomorphisms. In particular, $H$ permutes the connected components of this set. This gives us the following. 

\begin{theorem}\label{thm:hyp-endsofpairs}
Let $G$ be a one-ended hyperbolic group and $H$ a quasiconvex subgroup. Then $e(G,H)$ is equal to the number of $H$-orbits of components of $\partial G - \Lambda H$. 
\end{theorem}

\begin{proof}
Recall by Proposition~\ref{prop:coset-ends} that $e(G,H)$ is equal to the number of ends of the coset graph $H\backslash \Gamma$. Note that each right coset $Hg$ satisfies the following: $d(x,H)$ is constant across all $x \in Hg$. We call this value the \textit{height} of a coset. It is easy to see that at any given height there is only finitely many cosets.

Now, we will form the coset graph of $H$ by a sequence of identifications, in a way which will make the conclusion clear. First, identify all cosets which lie in $N_r(H)$ to points. Next, identify all points in $N_{r,R}(H) - A_{r,R,K}(H)$ which lie in the same coset. Now, in each component of $A_{r,\infty,K}(H)$, identify vertices which lie in the same coset of $H$. At this stage, it is clear we have a graph with $\tilde e(G,H)$ ends. The only remaining identifications to be made are identifying some of these ``ends''. We have that the number of ends of the final graph is then equal to the number of $H$-orbits of components of $A_{r,\infty,K}(H)$, which is equal to the number of $H$-orbits of components of $\partial G - \Lambda H$. But the number of ends of this graph is equal to $e(G,H)$ by Proposition~\ref{prop:coset-ends}, and so the result follows.
\end{proof}

The results of this section lead us to a straightforward proof of the following corollary.

\begin{corollary}\label{cor:ends-of-pair-finite}
Let $G$ be a one-ended hyperbolic group, and $H$ a quasiconvex subgroup. Then $e(G,H)$ is finite. In particular, $e(G,H) \leq |B_{2l + R}(1)|$, where $l$ is the longest length of a given generator of $H$. 
\end{corollary}

\begin{proof}
The action of $H$ on the components of $\partial G - \Lambda H$ is induced by the action of $H$ on the components of $A_{r,R,K}(H)$. The latter is a locally finite graph on which $H$ acts cocompactly, and so this must have finitely many $H$-orbits of connected components, since the quotient graph is finite. It follows that there can only be finitely many $H$-orbits of components of $\partial G - \Lambda H$, and the result follows by Theorem~\ref{thm:hyp-endsofpairs}. The precise bound follows from just bounding above the size of a fundamental domain for the action of $H$ on $A_{r,R,K}(H)$.
\end{proof}

The above results thus give us a full description of all possible $H$-almost invariant subsets of $G$, up to equivalence: any such set is equivalent to some union of $H$-orbits of components of $A_{r,\infty,K}(H)$. 
To ease notation, if $U$ is a union of connected components of $A_{r,\infty,K}(H)$, then let 
$$
U^\ast := A_{r,\infty,K}(H) - U. 
$$
It is easy to see that the symmetric difference $U^\ast \triangle (G - U)$ is $H$-finite, so this overloading of notation is not a problem for our purposes.  

Now, note that if $U$ is a connected component of $A_{r,\infty,K}(H)$, then $gU$ is a connected component of $A_{r,\infty,K}(gH)$. The connected components of the latter are in one-to-one correspondence with the components of $\partial G - \Lambda gH$, and we can define $\shad(gU)$ in the obvious way to realise this correspondence. Note that one has to be careful with the choice of basepoint, but this doesn't really matter as if we increase $r$ accordingly for each $g$, then it is easy to see that the choice of basepoint does not affect the properties of $\shad$.
We conclude this section by characterising via the boundary what it means for an $H$-almost invariant set $X$ to cross itself, in the sense of Definition~\ref{def:crossings}.

\begin{proposition}\label{prop:crossings-in-boundary}
Let $X \subset G$ be an $H$-almost invariant set, and let $U$ be the unique union of connected components of $A_{r,\infty,K}(H)$ which is equivalent to $X$. Then $X$ crosses itself if and only if there exists $g \in G$ such that 
$$
\shad U \cap \shad (gU^\ast), \ \ \ \shad U \cap \shad (gU), \ \ \ \shad (U^\ast) \cap \shad (gU), \ \ \ \shad (U^\ast) \cap \shad (gU^\ast)
$$ 
are all non-empty.
\end{proposition}

\begin{proof}
It can be seen that $U \cap gU$ is $H$-infinite if and only if it contains arbitrarily large balls \cite[Remark~1.13]{niblo2005minimal}. This clearly implies that $\shad U \cap \shad (gU) \neq \emptyset$. 

Conversely, $\shad U \cap \shad (gU)$ is an open subset of $\partial G$, so if it is non-empty then it contains some open ball $B \subset \shad U \cap \shad (gU)$. It is follows quickly from this that $ U \cap gU$ must contain balls of arbitrary diameter. By symmetry in the other four cases, the result follows. 
\end{proof}

\section{Detection via Digraphs}

So far, we have related the connectivity of the limit set complement to the connectivity of a certain subgraph of the Cayley graph, but this subgraph is still infinite. In order to obtain any sort of algorithm we will need some way of understanding the global connectivity properties by looking only at a finite piece.

In this section, we will make headway towards this goal by constructing a certain \textit{labelled digraph}, whose connectivity and language encodes a great deal of information about the properties of the subgraph $A_{r,R,K}(H)$.

\subsection{Digraphs and their languages}

We begin by briefly introducing labelled digraphs. This area has many applications to the study of subgroups of free groups, and a good survey of this rich theory can be found in \cite{kapovich2002stallings}. 

\begin{definition}[Labelled digraphs]
Let $S$ be a finite set of symbols, closed under taking formal inverses, i.e. $S = S^{-1}$. An \textit{$S$-digraph} $\Delta$ is a finite graph where every edge is oriented and is labelled by some $s \in S$. 

We also require that if there is a directed edge $e$ labelled by $s \in S$ from $v$ to $v'$, then there is an edge labelled by $s^{-1}$ from $v'$ to $v$, which we denote by $e^{-1}$.  
\end{definition}

Note that in this definition we allow the possibility of single-edge loops, and we do not require our graphs to be connected. We now set up some notation.  

We write $V\Delta$ for the vertex set of $\Delta$. 
Denote by $o(e)$ and $t(e)$ the initial and terminal vertices of the oriented edge $e$, respectively. Denote by $\Lab(e)$ the label of the oriented edge $e$, so $\Lab(e^{-1}) = \Lab(e)^{-1}$. A \textit{path} through $\Delta$ is a sequence of oriented edges $e_1e_2 \ldots e_n$ such that $o(e_{i+1}) = t(e_i)$. We call this path a \textit{loop} if $o(e_1) = t(e_n)$. We define the \textit{inverse} of the path $p$ as the path $p^{-1} := e^{-1}_n\ldots e^{-1}_1$. Given $p$ as above, define $o(p) = o(e_1)$, $t(p) = t(e_n)$. 
We say that two paths $p$, $q$ are \textit{freely equal} if $o(p) = o(q)$, $t(p) = t(q)$, and the words $\Lab(p)$ and $\Lab(q)$ are equal in the free group $F(S)$. 

Given a set of symbols $S$, denote by $S^\ast$ the set of finite strings in $S$. If $p = e_1 \ldots e_n$ is a path through $\Delta$, define its label $\Lab(p) \in S^\ast$ as the formal string
$
\Lab(p) = \Lab(e_1) \ldots \Lab(e_n)$. We have the following definition.

\begin{definition}
Let $\Delta$ be an $S$-digraph, and let $v$, $v'$ be vertices of $\Delta$. Define the \textit{language of $\Delta$ from $v$ to $v'$} as the set 
$$
\mathcal L (\Delta, v, v') = \set {\Lab(p)}{\text{$p$ is a path in $\Delta$ from $v$ to $v'$}} \subset S^\ast.
$$
\end{definition}

Let $H$ be a group and $S \subset H$ a finite symmetric subset. Denote by $\pi : S^\ast \to H$ the obvious projection. Then we have the following key lemma.

\begin{lemma}\label{lem:image-of-language}
Let $\Delta$ be an $S$-digraph, where $S$ is a finite subset of a group $H$.
Then for every pair $v_0, v_1$ of vertices of $\Delta$, there is a finitely generated subgroup $K_{v_0} \leq H$ and a finite subset $T_{v_0,v_1} \subset H$ such that 
$$
\pi(\mathcal L(\Delta, v_0, v_1)) = \bigcup_{t \in T_{v_0,v_1}} K_{v_0} t.
$$
Moreover, $T_{v_0,v_1}$ and generators of $K_{v_0}$ can be computed effectively. 
\end{lemma}

\begin{proof}
To ease notation, let $\mathcal L = \mathcal L(\Delta, v_0, v_1)$. We first construct $K_{v_0}$ and $T_{v_0,v_1}$.
Consider the following sets of paths through $\Delta$.
\begin{itemize}
    \item $\Pi(v, v') = \{p : \textrm{$p$ is a simple path through $\Delta$ with $o(p) = v$, $t(p) = v'$}\}$,
    \item $\Lambda(v) = \{\ell : \textrm{$\ell$ is a simple loop through $\Delta$ with $o(\ell) = t(\ell) = v$}\}$,
    \item $\Lambda'(v) = \{p\ell p^{-1} : v' \in V\Delta, \ \ell \in \Lambda(v'), \ p \in \Pi(v, v') \}$.
\end{itemize}
Intuitively, $\Lambda'(v)$ is the set of ``lollipop loops'' based at $v$, which are formed by attaching a simple loop to a simple path and its inverse. 
Let $Y = \{ \pi(\Lab(\gamma)) : \gamma \in \Lambda'(v_0)\}$, and set $K_{v_0} = \langle Y \rangle$. Also define
$$
T_{v_0, v_1} = \{ \pi(\Lab(p)) : p \in \Pi(v_0, v_1) \}.
$$
The rest of this proof will be dedicated to showing the following statement. We claim that every path $\gamma$ through $\Delta$ with $o(\gamma) = v_0$ is freely equal to a path of the form $\alpha_1 \ldots \alpha_r q$, where each $\alpha_i \in \Lambda(v_0)$ and $q \in \Pi(v_0, t(\gamma))$. From this the lemma follows immediately. 

With the above claim in mind, let $\gamma$ be a path in $\Delta$ from $v_0$ to $v_1$ and assume without loss of generality that $\Lab(\gamma)$ is freely reduced. If $\gamma$ is a simple path or simple loop then the statement is trivial, so assume $\gamma$ self-intersects somewhere away from $o(\gamma) = v_0$. We then proceed by induction on the length of $\gamma$. We may decompose $\gamma$ into the form
$$
\gamma = p_0 \ell_1 p_1 \ell_2 p_2 \ldots \ell_k p_k,
$$
where the $p_j$ are (possibly trivial) simple paths such that $t(p_j) = o(p_{j+1})$ and each $\ell_j$ is a non-trivial simple loop in $\Lambda(o(p_j))$. 
For each $1 \leq j \leq k$ let
$$
\ell'_j := p_0 p_1 \ldots p_{j-1} \ell_j p_{j-1}^{-1} \ldots p_1 ^{-1} p_0 ^{-1}. 
$$
Inspection reveals then that $\gamma$ is freely equal to the path $\gamma' := \ell_1' \ell_2' \ldots \ell_k'p_0p_1\ldots p_k$. 
Consider the subpath $\gamma_j = p_0 \ldots p_j$. Each $\gamma_j$ is a path of strictly shorter length than $\gamma$. By the inductive hypothesis we have that $\gamma_j$ is of the form $s_1 \ldots s_kq$ where each $s_i  \in \Lambda'(v_0)$ and $q \in \Pi(v_0, t(p_j))$. 
Now we have that 
$$
\ell_{j+1}' = \gamma_{j-1} \ell_j \gamma_{j-1}^{-1} = s_1 \ldots s_i q \ell_j q^{-1} s_i^{-1} \ldots s_1^{-1}.
$$
Since $q\ell_j q^{-1} \in \Lambda'(v_0)$, we have that $\ell_j'$ is thus freely equal to a path of the form $\alpha_1 \ldots \alpha_r$ where each $\alpha_i \in \Lambda'(v_0)$. If we apply the inductive hypothesis once more to $\gamma_k$, then the claim follows. 
\end{proof}

\subsection{Adjacency digraphs}

Let $G$ be a finitely generated group and $\Gamma$ a Cayley graph of $G$. 
Let $H$ be a finitely generated subgroup of $G$.
Let $F \subset \Gamma$ be a finite subgraph, and consider $X = {HF}$. Then $X$ is a subgraph of $\Gamma$ on which $H$ acts cocompactly. We say that $F$ is a \textit{finite $H$-cover} of $X$. In principle, $X$ will have many finite $H$-covers.

Given such a finite set of vertices $F$, we form its $H$-\textit{adjacency set}, or just its \textit{adjacency set} $S_F$ as follows. Let 
$$
S_F = \set {s \in H}{ s \neq 1, \ sF \cap F \neq \emptyset}.
$$
Informally, $S_F$ is the finite set containing $s\in H$ such that $sF$ intersects $F$. Note that $S_F$ is symmetric, i.e $S_F = S_F^{-1}$. The fact that $S_F$ is finite follows from $F$ being finite and $\Gamma$ being locally finite.  

 We form the $S_F$-digraph $\Delta_F$, called the $F$-\textit{adjacency graph}, as follows. The vertex set  $V\Delta_F$ is precisely the set $\pi_0 (F)$ of connected components of $F$, and there is a directed edge labelled by $s \in S_F$ from $v$ to $v'$ if $v \cap sv' \neq \emptyset$. In particular, this implies that $v$ and $sv'$ are contained in the same connected component of $X$. It is easy to check that $\Delta_F$ is indeed a well-defined digraph. See Figure~\ref{fig:adj-digraph} for an example of this construction. 

 \begin{figure}
     \centering

% Pattern Info
 
\tikzset{
pattern size/.store in=\mcSize, 
pattern size = 5pt,
pattern thickness/.store in=\mcThickness, 
pattern thickness = 0.3pt,
pattern radius/.store in=\mcRadius, 
pattern radius = 1pt}
\makeatletter
\pgfutil@ifundefined{pgf@pattern@name@_lch571ju3}{
\pgfdeclarepatternformonly[\mcThickness,\mcSize]{_lch571ju3}
{\pgfqpoint{0pt}{0pt}}
{\pgfpoint{\mcSize+\mcThickness}{\mcSize+\mcThickness}}
{\pgfpoint{\mcSize}{\mcSize}}
{
\pgfsetcolor{\tikz@pattern@color}
\pgfsetlinewidth{\mcThickness}
\pgfpathmoveto{\pgfqpoint{0pt}{0pt}}
\pgfpathlineto{\pgfpoint{\mcSize+\mcThickness}{\mcSize+\mcThickness}}
\pgfusepath{stroke}
}}
\makeatother

% Pattern Info
 
\tikzset{
pattern size/.store in=\mcSize, 
pattern size = 5pt,
pattern thickness/.store in=\mcThickness, 
pattern thickness = 0.3pt,
pattern radius/.store in=\mcRadius, 
pattern radius = 1pt}
\makeatletter
\pgfutil@ifundefined{pgf@pattern@name@_9mdfs4c3p}{
\pgfdeclarepatternformonly[\mcThickness,\mcSize]{_9mdfs4c3p}
{\pgfqpoint{0pt}{0pt}}
{\pgfpoint{\mcSize+\mcThickness}{\mcSize+\mcThickness}}
{\pgfpoint{\mcSize}{\mcSize}}
{
\pgfsetcolor{\tikz@pattern@color}
\pgfsetlinewidth{\mcThickness}
\pgfpathmoveto{\pgfqpoint{0pt}{0pt}}
\pgfpathlineto{\pgfpoint{\mcSize+\mcThickness}{\mcSize+\mcThickness}}
\pgfusepath{stroke}
}}
\makeatother

% Pattern Info
 
\tikzset{
pattern size/.store in=\mcSize, 
pattern size = 5pt,
pattern thickness/.store in=\mcThickness, 
pattern thickness = 0.3pt,
pattern radius/.store in=\mcRadius, 
pattern radius = 1pt}
\makeatletter
\pgfutil@ifundefined{pgf@pattern@name@_utyw2w9k3}{
\pgfdeclarepatternformonly[\mcThickness,\mcSize]{_utyw2w9k3}
{\pgfqpoint{0pt}{0pt}}
{\pgfpoint{\mcSize+\mcThickness}{\mcSize+\mcThickness}}
{\pgfpoint{\mcSize}{\mcSize}}
{
\pgfsetcolor{\tikz@pattern@color}
\pgfsetlinewidth{\mcThickness}
\pgfpathmoveto{\pgfqpoint{0pt}{0pt}}
\pgfpathlineto{\pgfpoint{\mcSize+\mcThickness}{\mcSize+\mcThickness}}
\pgfusepath{stroke}
}}
\makeatother

% Pattern Info
 
\tikzset{
pattern size/.store in=\mcSize, 
pattern size = 5pt,
pattern thickness/.store in=\mcThickness, 
pattern thickness = 0.3pt,
pattern radius/.store in=\mcRadius, 
pattern radius = 1pt}
\makeatletter
\pgfutil@ifundefined{pgf@pattern@name@_4eii05wxz}{
\pgfdeclarepatternformonly[\mcThickness,\mcSize]{_4eii05wxz}
{\pgfqpoint{0pt}{0pt}}
{\pgfpoint{\mcSize+\mcThickness}{\mcSize+\mcThickness}}
{\pgfpoint{\mcSize}{\mcSize}}
{
\pgfsetcolor{\tikz@pattern@color}
\pgfsetlinewidth{\mcThickness}
\pgfpathmoveto{\pgfqpoint{0pt}{0pt}}
\pgfpathlineto{\pgfpoint{\mcSize+\mcThickness}{\mcSize+\mcThickness}}
\pgfusepath{stroke}
}}
\makeatother

% Pattern Info
 
\tikzset{
pattern size/.store in=\mcSize, 
pattern size = 5pt,
pattern thickness/.store in=\mcThickness, 
pattern thickness = 0.3pt,
pattern radius/.store in=\mcRadius, 
pattern radius = 1pt}
\makeatletter
\pgfutil@ifundefined{pgf@pattern@name@_tpd0c0vdo}{
\pgfdeclarepatternformonly[\mcThickness,\mcSize]{_tpd0c0vdo}
{\pgfqpoint{0pt}{-\mcThickness}}
{\pgfpoint{\mcSize}{\mcSize}}
{\pgfpoint{\mcSize}{\mcSize}}
{
\pgfsetcolor{\tikz@pattern@color}
\pgfsetlinewidth{\mcThickness}
\pgfpathmoveto{\pgfqpoint{0pt}{\mcSize}}
\pgfpathlineto{\pgfpoint{\mcSize+\mcThickness}{-\mcThickness}}
\pgfusepath{stroke}
}}
\makeatother

% Pattern Info
 
\tikzset{
pattern size/.store in=\mcSize, 
pattern size = 5pt,
pattern thickness/.store in=\mcThickness, 
pattern thickness = 0.3pt,
pattern radius/.store in=\mcRadius, 
pattern radius = 1pt}
\makeatletter
\pgfutil@ifundefined{pgf@pattern@name@_8qjvz9jam}{
\pgfdeclarepatternformonly[\mcThickness,\mcSize]{_8qjvz9jam}
{\pgfqpoint{0pt}{-\mcThickness}}
{\pgfpoint{\mcSize}{\mcSize}}
{\pgfpoint{\mcSize}{\mcSize}}
{
\pgfsetcolor{\tikz@pattern@color}
\pgfsetlinewidth{\mcThickness}
\pgfpathmoveto{\pgfqpoint{0pt}{\mcSize}}
\pgfpathlineto{\pgfpoint{\mcSize+\mcThickness}{-\mcThickness}}
\pgfusepath{stroke}
}}
\makeatother

% Pattern Info
 
\tikzset{
pattern size/.store in=\mcSize, 
pattern size = 5pt,
pattern thickness/.store in=\mcThickness, 
pattern thickness = 0.3pt,
pattern radius/.store in=\mcRadius, 
pattern radius = 1pt}
\makeatletter
\pgfutil@ifundefined{pgf@pattern@name@_l0kbzia19}{
\pgfdeclarepatternformonly[\mcThickness,\mcSize]{_l0kbzia19}
{\pgfqpoint{0pt}{-\mcThickness}}
{\pgfpoint{\mcSize}{\mcSize}}
{\pgfpoint{\mcSize}{\mcSize}}
{
\pgfsetcolor{\tikz@pattern@color}
\pgfsetlinewidth{\mcThickness}
\pgfpathmoveto{\pgfqpoint{0pt}{\mcSize}}
\pgfpathlineto{\pgfpoint{\mcSize+\mcThickness}{-\mcThickness}}
\pgfusepath{stroke}
}}
\makeatother

% Pattern Info
 
\tikzset{
pattern size/.store in=\mcSize, 
pattern size = 5pt,
pattern thickness/.store in=\mcThickness, 
pattern thickness = 0.3pt,
pattern radius/.store in=\mcRadius, 
pattern radius = 1pt}
\makeatletter
\pgfutil@ifundefined{pgf@pattern@name@_1ym057evs}{
\pgfdeclarepatternformonly[\mcThickness,\mcSize]{_1ym057evs}
{\pgfqpoint{0pt}{-\mcThickness}}
{\pgfpoint{\mcSize}{\mcSize}}
{\pgfpoint{\mcSize}{\mcSize}}
{
\pgfsetcolor{\tikz@pattern@color}
\pgfsetlinewidth{\mcThickness}
\pgfpathmoveto{\pgfqpoint{0pt}{\mcSize}}
\pgfpathlineto{\pgfpoint{\mcSize+\mcThickness}{-\mcThickness}}
\pgfusepath{stroke}
}}
\makeatother
\tikzset{every picture/.style={line width=0.75pt}} %set default line width to 0.75pt        

\begin{tikzpicture}[x=0.75pt,y=0.75pt,yscale=-1,xscale=1]

%Shape: Polygon [id:ds10521538463356994] 
\draw  [color={rgb, 255:red, 208; green, 2; blue, 27 }  ,draw opacity=1 ][fill={rgb, 255:red, 208; green, 2; blue, 27 }  ,fill opacity=0.54 ] (100.29,58.02) -- (118.91,65.45) -- (121.93,93.21) -- (108.34,112.58) -- (91.74,113.87) -- (81.17,98.7) -- (81.92,82.56) -- (90.48,64.48) -- cycle ;
%Shape: Polygon [id:ds7403438089820804] 
\draw  [color={rgb, 255:red, 208; green, 2; blue, 27 }  ,draw opacity=1 ][fill={rgb, 255:red, 208; green, 2; blue, 27 }  ,fill opacity=0.54 ] (150.37,63.51) -- (168.24,67.71) -- (175.53,90.63) -- (161.69,106.44) -- (146.59,108.38) -- (134.77,96.11) -- (135.52,79.97) -- (146.34,81.59) -- cycle ;
%Shape: Polygon [id:ds264982739350085] 
\draw  [color={rgb, 255:red, 208; green, 2; blue, 27 }  ,draw opacity=1 ][fill={rgb, 255:red, 208; green, 2; blue, 27 }  ,fill opacity=0.54 ] (103.81,131.94) -- (113.13,145.18) -- (112.12,156.48) -- (106.58,159.38) -- (99.03,163.9) -- (88.72,163.9) -- (79.15,153.25) -- (87.21,137.75) -- cycle ;
%Shape: Polygon [id:ds9054821163793201] 
\draw  [color={rgb, 255:red, 208; green, 2; blue, 27 }  ,draw opacity=1 ][fill={rgb, 255:red, 208; green, 2; blue, 27 }  ,fill opacity=0.54 ] (165.97,131.94) -- (174.28,142.6) -- (168.49,157.77) -- (162.95,160.67) -- (155.4,165.19) -- (145.08,165.19) -- (136.53,144.53) -- (144.83,129.69) -- cycle ;
%Shape: Polygon [id:ds6937908685805834] 
\draw  [color={rgb, 255:red, 74; green, 144; blue, 226 }  ,draw opacity=1 ][pattern=_lch571ju3,pattern size=6pt,pattern thickness=0.75pt,pattern radius=0pt, pattern color={rgb, 255:red, 74; green, 144; blue, 226}] (178.86,27.16) -- (197.48,34.58) -- (200.5,62.34) -- (186.91,81.71) -- (170.31,83) -- (159.74,67.83) -- (160.49,51.69) -- (169.05,33.62) -- cycle ;
%Shape: Polygon [id:ds10825142290947243] 
\draw  [color={rgb, 255:red, 74; green, 144; blue, 226 }  ,draw opacity=1 ][pattern=_9mdfs4c3p,pattern size=6pt,pattern thickness=0.75pt,pattern radius=0pt, pattern color={rgb, 255:red, 74; green, 144; blue, 226}] (78.65,104.83) -- (96.52,109.03) -- (103.81,131.94) -- (89.97,147.76) -- (74.88,149.7) -- (63.05,137.43) -- (63.8,121.29) -- (74.62,122.91) -- cycle ;
%Shape: Polygon [id:ds21419174829323784] 
\draw  [color={rgb, 255:red, 74; green, 144; blue, 226 }  ,draw opacity=1 ][pattern=_utyw2w9k3,pattern size=6pt,pattern thickness=0.75pt,pattern radius=0pt, pattern color={rgb, 255:red, 74; green, 144; blue, 226}] (50.04,29.92) -- (59.35,43.16) -- (58.34,54.45) -- (52.81,57.36) -- (45.26,61.88) -- (34.94,61.88) -- (25.38,51.23) -- (33.43,35.73) -- cycle ;
%Shape: Polygon [id:ds36142044602640677] 
\draw  [color={rgb, 255:red, 74; green, 144; blue, 226 }  ,draw opacity=1 ][pattern=_4eii05wxz,pattern size=6pt,pattern thickness=0.75pt,pattern radius=0pt, pattern color={rgb, 255:red, 74; green, 144; blue, 226}] (168.07,161.52) -- (176.37,172.18) -- (170.58,187.35) -- (165.05,190.25) -- (157.5,194.77) -- (147.18,194.77) -- (138.62,174.11) -- (146.93,159.26) -- cycle ;
%Shape: Polygon [id:ds880711784760015] 
\draw  [color={rgb, 255:red, 126; green, 211; blue, 33 }  ,draw opacity=1 ][pattern=_tpd0c0vdo,pattern size=6pt,pattern thickness=0.75pt,pattern radius=0pt, pattern color={rgb, 255:red, 126; green, 211; blue, 33}] (66.7,57.13) -- (83.31,63.76) -- (86,88.55) -- (73.88,105.85) -- (59.07,107) -- (49.64,93.45) -- (50.32,79.04) -- (57.95,62.89) -- cycle ;
%Shape: Polygon [id:ds08020887957656364] 
\draw  [color={rgb, 255:red, 126; green, 211; blue, 33 }  ,draw opacity=1 ][pattern=_8qjvz9jam,pattern size=6pt,pattern thickness=0.75pt,pattern radius=0pt, pattern color={rgb, 255:red, 126; green, 211; blue, 33}] (109.51,27.68) -- (127.38,31.87) -- (134.68,54.79) -- (120.84,70.61) -- (105.74,72.55) -- (93.91,60.28) -- (94.67,44.14) -- (105.49,45.75) -- cycle ;
%Shape: Polygon [id:ds5491614907139164] 
\draw  [color={rgb, 255:red, 126; green, 211; blue, 33 }  ,draw opacity=1 ][pattern=_l0kbzia19,pattern size=6pt,pattern thickness=0.75pt,pattern radius=0pt, pattern color={rgb, 255:red, 126; green, 211; blue, 33}] (167.72,101.25) -- (177.03,114.49) -- (176.02,125.79) -- (170.49,128.69) -- (162.94,133.21) -- (152.62,133.21) -- (143.06,122.56) -- (151.11,107.06) -- cycle ;
%Shape: Polygon [id:ds42618834461547195] 
\draw  [color={rgb, 255:red, 126; green, 211; blue, 33 }  ,draw opacity=1 ][pattern=_1ym057evs,pattern size=6pt,pattern thickness=0.75pt,pattern radius=0pt, pattern color={rgb, 255:red, 126; green, 211; blue, 33}] (52.12,135.48) -- (60.43,146.13) -- (54.64,161.31) -- (49.1,164.21) -- (41.55,168.73) -- (31.24,168.73) -- (22.68,148.07) -- (30.98,133.22) -- cycle ;
%Curve Lines [id:da5061512378576825] 
\draw [color={rgb, 255:red, 74; green, 144; blue, 226 }  ,draw opacity=1 ][line width=1.5]    (369.5,52.88) .. controls (385.5,69.67) and (392.64,73.25) .. (423.18,71.32) ;
\draw [shift={(426.08,71.12)}, rotate = 175.95] [color={rgb, 255:red, 74; green, 144; blue, 226 }  ,draw opacity=1 ][line width=1.5]    (14.21,-4.28) .. controls (9.04,-1.82) and (4.3,-0.39) .. (0,0) .. controls (4.3,0.39) and (9.04,1.82) .. (14.21,4.28)   ;
%Curve Lines [id:da6630948095266735] 
\draw [color={rgb, 255:red, 74; green, 144; blue, 226 }  ,draw opacity=1 ][line width=1.5]    (375.57,43.3) .. controls (404.68,47.7) and (415.19,53.56) .. (427.02,59.43) ;
\draw [shift={(372.32,42.83)}, rotate = 8.01] [color={rgb, 255:red, 74; green, 144; blue, 226 }  ,draw opacity=1 ][line width=1.5]    (14.21,-4.28) .. controls (9.04,-1.82) and (4.3,-0.39) .. (0,0) .. controls (4.3,0.39) and (9.04,1.82) .. (14.21,4.28)   ;
%Curve Lines [id:da7750309518768552] 
\draw [color={rgb, 255:red, 74; green, 144; blue, 226 }  ,draw opacity=1 ][line width=1.5]    (341.21,132.86) .. controls (358.21,115.07) and (398.08,94.18) .. (427.6,76.24) ;
\draw [shift={(429.85,74.86)}, rotate = 148.45] [color={rgb, 255:red, 74; green, 144; blue, 226 }  ,draw opacity=1 ][line width=1.5]    (14.21,-4.28) .. controls (9.04,-1.82) and (4.3,-0.39) .. (0,0) .. controls (4.3,0.39) and (9.04,1.82) .. (14.21,4.28)   ;
%Curve Lines [id:da898517221511878] 
\draw [color={rgb, 255:red, 74; green, 144; blue, 226 }  ,draw opacity=1 ][line width=1.5]    (433.42,173.29) .. controls (449.79,228.49) and (486.89,199.46) .. (440.22,162.79) ;
\draw [shift={(432.44,169.81)}, rotate = 75.22] [color={rgb, 255:red, 74; green, 144; blue, 226 }  ,draw opacity=1 ][line width=1.5]    (14.21,-4.28) .. controls (9.04,-1.82) and (4.3,-0.39) .. (0,0) .. controls (4.3,0.39) and (9.04,1.82) .. (14.21,4.28)   ;
%Curve Lines [id:da6295030767030243] 
\draw [color={rgb, 255:red, 126; green, 211; blue, 33 }  ,draw opacity=1 ][line width=1.5]    (443.47,51.97) .. controls (427.23,39.7) and (393.38,17.5) .. (358.89,29.5) ;
\draw [shift={(445.88,53.82)}, rotate = 217.77] [color={rgb, 255:red, 126; green, 211; blue, 33 }  ,draw opacity=1 ][line width=1.5]    (14.21,-4.28) .. controls (9.04,-1.82) and (4.3,-0.39) .. (0,0) .. controls (4.3,0.39) and (9.04,1.82) .. (14.21,4.28)   ;
%Curve Lines [id:da6799617549290067] 
\draw [color={rgb, 255:red, 126; green, 211; blue, 33 }  ,draw opacity=1 ][line width=1.5]    (349.13,142.97) .. controls (400.79,128.4) and (416.62,105.21) .. (438.57,79.54) ;
\draw [shift={(345.92,143.85)}, rotate = 345.03] [color={rgb, 255:red, 126; green, 211; blue, 33 }  ,draw opacity=1 ][line width=1.5]    (14.21,-4.28) .. controls (9.04,-1.82) and (4.3,-0.39) .. (0,0) .. controls (4.3,0.39) and (9.04,1.82) .. (14.21,4.28)   ;
%Curve Lines [id:da28574951732624054] 
\draw [color={rgb, 255:red, 126; green, 211; blue, 33 }  ,draw opacity=1 ][line width=1.5]    (345.16,155.3) .. controls (383.56,167.09) and (391.43,168.65) .. (419.48,159.98) ;
\draw [shift={(342.15,154.37)}, rotate = 17.11] [color={rgb, 255:red, 126; green, 211; blue, 33 }  ,draw opacity=1 ][line width=1.5]    (14.21,-4.28) .. controls (9.04,-1.82) and (4.3,-0.39) .. (0,0) .. controls (4.3,0.39) and (9.04,1.82) .. (14.21,4.28)   ;
%Curve Lines [id:da16157942825965255] 
\draw [color={rgb, 255:red, 126; green, 211; blue, 33 }  ,draw opacity=1 ][line width=1.5]    (343.2,49.85) .. controls (302.85,59.58) and (309.3,88.95) .. (352.05,54.75) ;
\draw [shift={(346.39,49.14)}, rotate = 168.54] [color={rgb, 255:red, 126; green, 211; blue, 33 }  ,draw opacity=1 ][line width=1.5]    (14.21,-4.28) .. controls (9.04,-1.82) and (4.3,-0.39) .. (0,0) .. controls (4.3,0.39) and (9.04,1.82) .. (14.21,4.28)   ;
%Shape: Ellipse [id:dp3328136737391474] 
\draw  [fill={rgb, 255:red, 225; green, 225; blue, 225 }  ,fill opacity=1 ] (345.45,42.83) .. controls (345.45,35.47) and (351.47,29.5) .. (358.89,29.5) .. controls (366.31,29.5) and (372.32,35.47) .. (372.32,42.83) .. controls (372.32,50.19) and (366.31,56.16) .. (358.89,56.16) .. controls (351.47,56.16) and (345.45,50.19) .. (345.45,42.83) -- cycle ;
%Shape: Ellipse [id:dp5888103945468885] 
\draw  [fill={rgb, 255:red, 225; green, 225; blue, 225 }  ,fill opacity=1 ] (425.13,66.21) .. controls (425.13,58.85) and (431.15,52.88) .. (438.57,52.88) .. controls (445.99,52.88) and (452.01,58.85) .. (452.01,66.21) .. controls (452.01,73.57) and (445.99,79.54) .. (438.57,79.54) .. controls (431.15,79.54) and (425.13,73.57) .. (425.13,66.21) -- cycle ;
%Shape: Ellipse [id:dp809128710922751] 
\draw  [fill={rgb, 255:red, 225; green, 225; blue, 225 }  ,fill opacity=1 ] (319.04,143.85) .. controls (319.04,136.49) and (325.06,130.52) .. (332.48,130.52) .. controls (339.9,130.52) and (345.92,136.49) .. (345.92,143.85) .. controls (345.92,151.21) and (339.9,157.18) .. (332.48,157.18) .. controls (325.06,157.18) and (319.04,151.21) .. (319.04,143.85) -- cycle ;
%Shape: Ellipse [id:dp00972895022604936] 
\draw  [fill={rgb, 255:red, 225; green, 225; blue, 225 }  ,fill opacity=1 ] (419,156.48) .. controls (419,149.11) and (425.02,143.15) .. (432.44,143.15) .. controls (439.86,143.15) and (445.88,149.11) .. (445.88,156.48) .. controls (445.88,163.84) and (439.86,169.81) .. (432.44,169.81) .. controls (425.02,169.81) and (419,163.84) .. (419,156.48) -- cycle ;

% Text Node
\draw (91.49,77.79) node [anchor=north west][inner sep=0.75pt]  [font=\small,color={rgb, 255:red, 208; green, 2; blue, 27 }  ,opacity=1 ]  {$A$};
% Text Node
\draw (149.85,83.05) node [anchor=north west][inner sep=0.75pt]  [font=\small,color={rgb, 255:red, 208; green, 2; blue, 27 }  ,opacity=1 ]  {$B$};
% Text Node
\draw (95.22,142.26) node [anchor=north west][inner sep=0.75pt]  [font=\small,color={rgb, 255:red, 208; green, 2; blue, 27 }  ,opacity=1 ]  {$C$};
% Text Node
\draw (146.72,135.89) node [anchor=north west][inner sep=0.75pt]  [font=\small,color={rgb, 255:red, 208; green, 2; blue, 27 }  ,opacity=1 ]  {$D$};
% Text Node
\draw (169.18,46.92) node [anchor=north west][inner sep=0.75pt]  [font=\footnotesize,color={rgb, 255:red, 74; green, 144; blue, 226 }  ,opacity=1 ]  {$aA$};
% Text Node
\draw (76.27,120.48) node [anchor=north west][inner sep=0.75pt]  [font=\footnotesize,color={rgb, 255:red, 74; green, 144; blue, 226 }  ,opacity=1 ]  {$aB$};
% Text Node
\draw (33.56,37.74) node [anchor=north west][inner sep=0.75pt]  [font=\footnotesize,color={rgb, 255:red, 74; green, 144; blue, 226 }  ,opacity=1 ]  {$aC$};
% Text Node
\draw (149.2,171.88) node [anchor=north west][inner sep=0.75pt]  [font=\footnotesize,color={rgb, 255:red, 74; green, 144; blue, 226 }  ,opacity=1 ]  {$aD$};
% Text Node
\draw (55.31,73.79) node [anchor=north west][inner sep=0.75pt]  [font=\footnotesize,color={rgb, 255:red, 126; green, 211; blue, 33 }  ,opacity=1 ]  {$bA$};
% Text Node
\draw (107.11,40.21) node [anchor=north west][inner sep=0.75pt]  [font=\footnotesize,color={rgb, 255:red, 126; green, 211; blue, 33 }  ,opacity=1 ]  {$bB$};
% Text Node
\draw (153.11,110.46) node [anchor=north west][inner sep=0.75pt]  [font=\footnotesize,color={rgb, 255:red, 126; green, 211; blue, 33 }  ,opacity=1 ]  {$bC$};
% Text Node
\draw (32.37,146.93) node [anchor=north west][inner sep=0.75pt]  [font=\footnotesize,color={rgb, 255:red, 126; green, 211; blue, 33 }  ,opacity=1 ]  {$bD$};
% Text Node
\draw (351.62,34.53) node [anchor=north west][inner sep=0.75pt]    {$A$};
% Text Node
\draw (431.86,58.38) node [anchor=north west][inner sep=0.75pt]    {$B$};
% Text Node
\draw (326.66,136.48) node [anchor=north west][inner sep=0.75pt]    {$C$};
% Text Node
\draw (426.12,149.58) node [anchor=north west][inner sep=0.75pt]    {$D$};
% Text Node
\draw (233.73,97.56) node [anchor=north west][inner sep=0.75pt]  [font=\Large]  {$\Longrightarrow $};
% Text Node
\draw (303.53,177.14) node [anchor=north west][inner sep=0.75pt]  [font=\normalsize]  {$\Delta _{F}$};
% Text Node
\draw (389.9,56.17) node [anchor=north west][inner sep=0.75pt]  [font=\footnotesize,color={rgb, 255:red, 74; green, 144; blue, 226 }  ,opacity=1 ]  {$a$};
% Text Node
\draw (360.2,98.26) node [anchor=north west][inner sep=0.75pt]  [font=\footnotesize,color={rgb, 255:red, 74; green, 144; blue, 226 }  ,opacity=1 ]  {$a$};
% Text Node
\draw (397.92,35.59) node [anchor=north west][inner sep=0.75pt]  [font=\footnotesize,color={rgb, 255:red, 74; green, 144; blue, 226 }  ,opacity=1 ]  {$a$};
% Text Node
\draw (464.4,185.25) node [anchor=north west][inner sep=0.75pt]  [font=\footnotesize,color={rgb, 255:red, 74; green, 144; blue, 226 }  ,opacity=1 ]  {$a$};
% Text Node
\draw (379.06,116.97) node [anchor=north west][inner sep=0.75pt]  [font=\footnotesize,color={rgb, 255:red, 126; green, 211; blue, 33 }  ,opacity=1 ]  {$b$};
% Text Node
\draw (385.19,150.64) node [anchor=north west][inner sep=0.75pt]  [font=\footnotesize,color={rgb, 255:red, 126; green, 211; blue, 33 }  ,opacity=1 ]  {$b$};
% Text Node
\draw (404.52,16.88) node [anchor=north west][inner sep=0.75pt]  [font=\footnotesize,color={rgb, 255:red, 126; green, 211; blue, 33 }  ,opacity=1 ]  {$b$};
% Text Node
\draw (303.61,59.91) node [anchor=north west][inner sep=0.75pt]  [font=\footnotesize,color={rgb, 255:red, 126; green, 211; blue, 33 }  ,opacity=1 ]  {$b$};
% Text Node
\draw (26.03,183.14) node [anchor=north west][inner sep=0.75pt]  [font=\normalsize]  {$F\cup aF\cup bF$};

\end{tikzpicture}

     \caption{An example of the construction of an adjacency digraph. Here, $\pi_0(F) = \{A,B, C,D\}$ and $S_F = \{a, b, a^{-1}, b^{-1}\}$. Due to symmetry we have not drawn the translates of $F$ by $a^{-1}$ or $b^{-1}$, nor the corresponding edges of $\Delta_F$.} 
     \label{fig:adj-digraph}
 \end{figure}
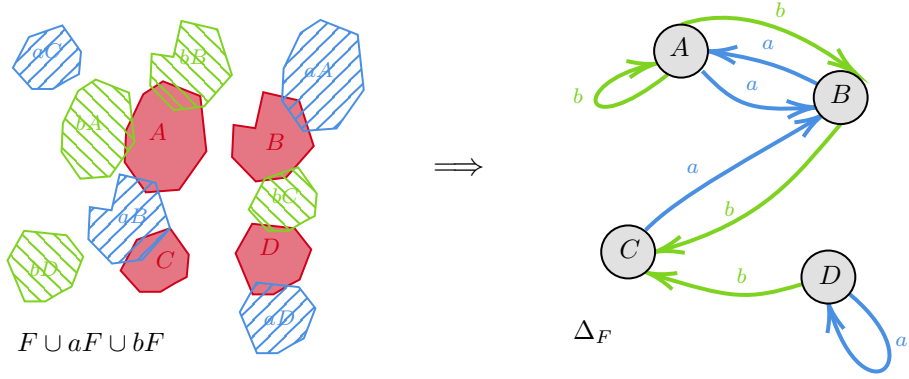
 
 The relevance of the adjacency digraph $\Delta_F$ to our problem is given by the following lemma.

\begin{lemma}\label{lem:adj-digraph}
Let $v, v' \in \pi_0 (F)$, and $h \in H$. Then $v$ and $hv'$ are contained in the same connected component of $X = {HF}$ if and only if $h \in \pi(\mathcal L(\Delta_F, v, v'))$. 
\end{lemma}

\begin{proof}
Firstly, suppose that $h = \pi (w)$, where $w=s_1\ldots s_n \in \mathcal L(\Delta_F, v, v')$. Let $\gamma = e_1\ldots e_n$ be a path through $\Delta_F$ between $v$ and $v'$ labelled by $w$, so $e_i$ is labelled by $s_i$. Let $v = v_0$ and $v_{i} = t(e_{i})$ for each $1\leq i \leq n$. Observe that by definition we have that $v_i$ intersects $s_{i+1}v_{i+1}$ in $\Gamma$, and so $s_{i+1}v_{i+1}$ and $v_i$ must be contained within the same component of $X$. 
Since $w \in \mathcal L(\Delta_F, v, v')$, we have that $v_n = v'$.
Note also that $s_{i+1}\ldots s_n \in \mathcal L (\Delta_F, v_i, v_n)$.
We proceed by induction on $n$. Clearly if $n = 1$, then the result is true by the definition of $\Delta_F$. 

Suppose then that $n > 1$. Then $w' = s_{2}\ldots s_n \in \mathcal L (\Delta_F, v_2, v_n)$. Let $h' = \pi(w')$, then by the inductive hypothesis we have that $v_1$ and $h'v_n$ are contained within the same component of $X$. Then $v_0$ is contained within the same component as $s_1 v_1$ by definition. But this is contained in the same component as $s_1h'v_n = hv'$. 

Conversely, suppose that $v$ and $hv'$ lie in the same component of $X$ for some $h \in H$. Let $\gamma$ be a path through $X$ between $v$ and $hv'$. The path $\gamma$ will pass through a sequence of translates of components of $F$, say 
$$
v = h_0v_0, \  h_1v_1, \ h_2 v_2, \ \ldots , \ h_lv_l = hv',
$$
where $h_i \in H$, $h_0 = 1$, $v_i \in \pi_0 (F)$, and $h_{i-1}v_{i-1} \cap h_{i}v_{i} \neq \emptyset$ for every $1\leq i\leq n$.
By definition, we then have that $s_{i} := h_{i-1}^{-1}h_{i} \in S_F$, and there is an edge in $\Delta_F$ from $v_{i-1}$ to $v_i$ labelled by $s_i$. We thus deduce that $s_1\ldots s_n \in \mathcal L(\Delta_F, v, v')$. Since $h = s_1 \ldots s_n$, the result follows. 
\end{proof}

The next result is a characterisation of when the graph $X = HF$ has only finitely many components. Given $v, v' \in \pi_0(F)$, recall the definition of $K_v$ and $T_{v,v'}$ from Lemma~\ref{lem:image-of-language}. 

\begin{lemma}\label{lem:fin-components}
Let $F$ be a finite $H$-cover of $X =  {HF}$. Then $X$ has finitely many connected components if and only if for every $v \in \pi_0 (F)$ we have that $K_{v}$ has finite index in $H$. 
\end{lemma}

\begin{proof}
By Lemmas~\ref{lem:image-of-language} and \ref{lem:adj-digraph}, note that $h \in K_v$ if and only if there is a path through $X$ from $v$ to $hv$, and thus the same also holds for $h'v$ and $h'hv$ for any $h' \in H$. So suppose first that some $K_v$ has finite index in $H$ and let $T$ be a finite left-transversal of $K_v$. Then $TK_v v = Hv$, and this intersects at most $|T|$ components of $X$. Repeat this for every $v \in \pi_0(F)$ and see that $X$ has at most 
$$
\sum_{v\in \pi_0(F)} |H : K_v| < \infty
$$
components.

Conversely, fix $v \in \pi_0(F)$ such that $|H: K_v| = \infty$ and let $T = \{h_i : i \in \N\}$ be an infinite left-transversal in $H$ of $K_v$. We claim that every $h_iv$ lies in a distinct component of $X$. Indeed, suppose that $h_iv$ and $h_jv$ were contained in the same component of $X$ for some $i \neq j$. Then by Lemmas~\ref{lem:image-of-language} and \ref{lem:adj-digraph} we have that $h_i^{-1}h_j \in K_v$. This contradicts our choice of $T$, so the lemma follows. 
\end{proof}

It is clear that $H$ permutes the connected components of $X$. Note that if $X$ is connected then $\Delta_F$ is certainly connected, though the converse is not necessarily true. Instead, the number of connected components of $\Delta_F$ actually encodes the following. 

\begin{lemma}\label{lem:comp-is-coends}
The number of connected components of $\Delta_F$ is equal to the number of $H$-orbits of components of $X = {HF}$. 
\end{lemma}

\begin{proof}
By Lemma~\ref{lem:adj-digraph} we see that 
$v, v'\in \pi_0 (F)$ are connected by a path in $\Delta_F$ if and only if  $v$ is joined by a path through $X$ to some $H$-translate of $v'$. The lemma then follows immediately.
\end{proof}

\subsection{First algorithms}

The first application of our digraph machinery is the following. Throughout this subsection we fix constants $0\leq r \leq K \leq R < \infty$ and a Cayley graph $\Gamma$ of our one-ended hyperbolic group $G$. 

\begin{proposition}\label{prop:decide-arrk}
Let $H$ be a quasiconvex subgroup of a one-ended hyperbolic group $G$. Then there is an algorithm which, given $x \in G$, will decide if $x \in A_{r,R,K}(H)$.
\end{proposition}

\begin{proof}
First, we show that one can decide membership of $N_{r,R}(H)$. Given $x \in G$, compute the finite balls $U_1 = B_r(x)$, $U_2 = B_R(x)$ in $\Gamma$. Then $x \in N_{r,R}(H)$ if and only if $U_1 \cap H = \emptyset$ and $U_2 \cap H \neq \emptyset$. This is decidable, since membership of $H$ is decidable. Similarly, we can decide membership of $C_K(H)$. If $x$ is not in $N_{r,R}(H)$, then terminate and return `no'.

% Write $N := N_{r,R}(H)$ to ease notation.
Now, note that $H$ acts cocompactly on $N_{r,R}(H)$, so compute a finite $H$-cover $F$ containing $x$. This can be achieved by, for example, letting $l$ be the length of the longest given generator of $H$, and choosing $F = N_{r,R}(H) \cap B_{2l+R}(x)$. Let $v_x \in \pi_0 (F)$ be the component of $F$ containing $x$. We now form the adjacency digraph $\Delta_F$, and mark the vertices $v \in \pi_0 (F)$ which intersect $C_K(H)$ in $\Gamma$. One can then check that $x \in A_{r,R,K}(H)$ if and only if $v_x$ lies in a connected component of $\Delta_F$ containing a marked vertex. 
\end{proof}

We now have the following algorithms, which will allow us to distinguish (filtered) ends from one another. 

\begin{proposition}
There is an algorithm which, given $x, y \in A_{r,R,K}(H)$, will decide if $x$ and $y$ lie in distinct $H$-orbits of connected components of $A_{r,R,K}(H)$. 
\end{proposition} 

\begin{proof}
Use Proposition~\ref{prop:decide-arrk} to compute a finite $H$-cover $F$ of $A_{r,R,K}(H)$, and let $v_x, v_y \in \pi_0(F)$ be such that $x \in v_x$, $y \in v_y$. 
Then, as remarked in the proof of Lemma~\ref{lem:comp-is-coends}, we need only form the adjacency digraph $\Delta_F$ and check whether $v_x$ and $v_y$ lie in the same connected component of $\Delta_F$. 
\end{proof}

\begin{proposition}\label{prop:distinct-arrk-gwp}
There is an algorithm which, given $x, y \in A_{r,R,K}(H)$ and a solution to the generalised word problem for $H$, will decide if $x$ and $y$ lie in distinct connected components of $A_{r,R,K}(H)$. 
\end{proposition}

\begin{proof}
Find a finite $H$-cover $F$ for $A_{r,R,K}(H)$ containing both $x$ and $y$. Let $v_x, v_y \in \pi_0 (F)$ be the components of $F$ containing $x$ and $y$ respectively. By Lemma~\ref{lem:adj-digraph}, we have that $x$ and $y$ are contained in the same component if and only if $1 \in \pi(\mathcal L(\Delta_F, v_x, v_y))$. Form the subgroup $K_{v_x}$ and set of words $T_{v_x,v_y}$ as in Lemma~\ref{lem:image-of-language}, then $x$ and $y$ lie in different components if and only if $T_{v_x,v_y} \cap K_{v_x}$ is empty. Using our given solution to the generalised word problem in $H$, this is decidable. 
\end{proof}

We now turn to reproving Proposition~\ref{prop:distinct-arrk-gwp}, but we drop the hypothesis that $H$ has a solvable generalised word problem and replace it with the condition that $H$ has finitely many filtered ends in $G$. 
The key observation that makes this problem tractable in this case is the following. 

% \begin{lemma}\label{lem:fin-coends-attach}
% There exists a finite subset $F \subset A_{r,R,K}(H)$ such that every connected component of $A_{r,R,K}(H)$ intersects $F$. Moreover, we can effectively find such an $F$. 
% \end{lemma}

% \begin{proof}
% Let $Y$ be the given generators of $H$. Let $D \subset A_{r,R,K}(H)$ be any finite $H$-cover for the action of $H$ on $A_{r,R,K}(H)$. Let $F = YD \cup D$. It is easy to check that $F$ satisfies our requirements, In particular, if $F$ does not meet our requirements, then either $Y$ cannot be a generating set or $A_{r,R,K}(H)$ has infinitely many components.
% This construction is clearly effective. 
% \end{proof}

% \begin{lemma}\label{lem:kf-fi}
% With $F$ as in Lemma~\ref{lem:fin-coends-attach}, then for any $v \in \pi_0 (F)$ the subgroup $K_{F,v}$ has finite index in $H$.
% \end{lemma}

% \begin{proof}
% It is clear that via Lemma~\ref{lem:fin-coends-attach} we also have that every component of $A_{r,R,K}(H)$ intersects $hF$, for any $h$. 
% Since every component of $A_{r,R,K}(H)$ intersects every $H$-translate of $F$, we have for every $v \in \pi_0 (F)$ and every $h \in H$, there is some $v' \in \pi_0 (F)$ such that $v$ and $hv'$ lie in the same connected component of $A_{r,R,K}(H)$. It follows by Lemma~\ref{lem:adj-digraph} that the set of all labels of simple paths leaving $v$ in $\Delta_F$ is a transversal of $K_{F,v}$ in $H$. Since this set is finite the result follows. 
% \end{proof}

\begin{proposition}\label{prop:distinct-arrk-finite}
If $A_{r,R,K}(H)$ has finitely many components, then
there is an algorithm which, given $x, y \in A_{r,R,K}(H)$, will decide if $x$ and $y$ lie in distinct connected components of $A_{r,R,K}(H)$. 
\end{proposition}

\begin{proof}
It is known by Lemma~\ref{lem:fin-components} that, for every $v \in \pi_0 (F)$, the subgroup $K_v\leq H$ constructed in Lemma~\ref{lem:image-of-language} has finite index in $H$. In particular, $K_v$ is quasiconvex in $H$, and so we can decide membership of $K_v$. The algorithm then proceeds exactly as in Proposition~\ref{prop:distinct-arrk-gwp}. 
\end{proof}

\subsection{Counting (filtered) ends of pairs}\label{sec:counting-ends}

We conclude this section by giving algorithms to compute $e(G,H)$ and $\tilde e(G,H)$. Fix constants $0\leq r \leq K \leq R < \infty$ such that Proposition~\ref{prop:arrk-comp-corr} is satisfied. 
Let $F$ be a finite $H$-cover for $A_{r,R,K}(H)$, and form the adjacency digraph $\Delta_F$. Note that the construction of this digraph is completely effective. We immediately have the following new algorithm for computing $e(G,H)$, which was first shown to be computable by Vonseel \cite{vonseel2018ends}. 

\begin{theorem}\label{thm:decide-coends-scott}
There is an algorithm which, upon input of a one-ended hyperbolic group $G$ and generators of a quasiconvex subgroup $H$, will output $e(G,H)$.
\end{theorem}

\begin{proof}
One simply counts connected components of $\Delta_F$.
The construction of $\Delta_F$ is completely effective, so the result follows immediately from Lemma~\ref{lem:comp-is-coends} and Theorem~\ref{thm:hyp-endsofpairs}.
\end{proof}

Computing $\tilde e(G,H)$ poses more problems. In particular, we need to somehow be able to decide if $K_v$ has finite index in $H$, which can be seen to be undecidable for an arbitrary choice of $H$ via the Rips Construction (see e.g. \cite{baumslag1994unsolvable}). Thus, in light of Lemma~\ref{lem:fin-components} we cannot expect to be able to decide if $\tilde e(G,H) = \infty$ without either adding further hypotheses to $H$ (e.g. this is decidable if $H$ is free \cite{kapovich2002stallings}), or somehow further controlling the structure of $A_{r,R,K}(H)$. It is not clear whether the latter of these is even possible, which suggests this problem may be undecidable for arbitrary choices of quasiconvex $H \leq G$. We can however at least give the following two algorithms.

\begin{theorem}\label{thm:decide-limsetcomp-discon}
There is an algorithm which, upon input of a one-ended hyperbolic group $G$, generators of a quasiconvex subgroup $H$, a solution to the generalised word problem for $H$, and an integer $N \geq 0$, will decide whether $\tilde e(G,H) \geq N$. In particular, we can decide if $\partial G - \Lambda H$ is connected. 
\end{theorem}

\begin{proof}
Compute a finite $H$-cover $F_0$ of $A_{r,R,K}(H)$, then inductively define $F_{i+1} := YF_i \cup F_i$,
where $Y$ is a symmetric generating set for $H$. Thus we have an increasing sequence of $H$-covers $(F_i)$, where each strictly contains the last. For each $i$, let $N_i$ denote the number of components of $A_{r,R,K}(H)$ which intersect $F_i$. This number is computable by Proposition~\ref{prop:distinct-arrk-gwp}, and we can conclude that $e(G,H) \geq N_i$.
If there is some $i$ such that $N_i = N_{i+1}$, then since $Y$ is a generating set, it follows that $N_j = N_i$ for all $j > i$. 
Given $N$ as input, our algorithm will run until $N_i \geq N$ for some $i$, or terminate if the sequence $(N_i)_i$ stabilises. By the above, this will always halt. 
% The result then follows from Proposition~\ref{prop:arrk-comp-corr} and Theorem~\ref{thm:hyp-filteredends}. 
% Compute a finite $H$-cover $F$ of $A_{r,R,K}(H)$. Let $F' = YF \cup F$, where $Y$ is a generating set of $H$. Fix a point $x \in F'$, and for every other $y \in F'$ check if $x$ and $y$ lie in different components of $A_{r,R,K}(H)$ using Proposition~\ref{prop:distinct-arrk-gwp}. If some $y$ does lie in a different component to $x$ then $A_{r,R,K}(H)$ has multiple components. Otherwise it is easy to see that since $Y$ is a generating set we must have that $A_{r,R,K}(H)$ is connected. 
\end{proof}

\begin{theorem}\label{thm:decide-filtered-ends-finite}
There is an algorithm which, upon input of a one-ended hyperbolic group $G$ and generators of a quasiconvex subgroup $H$, will terminate if and only if $\tilde e(G,H)$ is finite. 
Moreover, upon termination it will output the value of $\tilde e(G,H)$. 
\end{theorem}

\begin{proof}
We proceed as before and compute a finite $H$-cover $F_0$ of $A_{r,R,K}(H)$, then inductively define $F_{i+1} := YF_i \cup F_i$,
where $Y$ is a symmetric generating set for $H$. Thus, we have an increasing sequence of $H$-covers
$$
F_0 \subset F_1 \subset F_2 \subset \ldots, 
$$
where each strictly contains the last. Moreover, $\bigcup_i F_i = A_{r,R,K}(H)$. For each $i \geq 0$ we now run the following search. 
For each component of $YF_i$, search for a path through $A_{r,R,K}(H)$ into $F_i$. If this process terminates for a given $i \geq 0$ then it follows from an easy induction argument that there is a path from any point in $A_{r,R,K}(H)=HF_i$ back to $F_i$ travelling through $A_{r,R,K}(H)$. In particular, this means that every connected component of $A_{r,R,K}(H)$ intersects $F_i$. 

Clearly such an $F_i$ exists if and only if $A_{r,R,K}(H)$ has finitely many components, which is equivalent to the condition that $\tilde e(G,H) < \infty$ by Proposition~\ref{prop:arrk-comp-corr} and Theorem~\ref{thm:hyp-filteredends}. 
If we do find such an $F_i$ then to compute the exact value of $\tilde e(G,H)$ we may use Proposition~\ref{prop:distinct-arrk-finite} to decide how many distinct components of $A_{r,R,K}(H)$ intersect $F_i$. By our choice of $F_i$, this will then be precisely the total number of components of $A_{r,R,K}(H)$.
\end{proof}

\section{Searching for Splittings and Crossings}

In this section we apply the above tools to the problem of deciding if a given quasiconvex subgroup is associated with a splitting. In short, we run two searches in parallel -- one search for a splitting and another search for obstructions to splittings.

\subsection{An algorithm to search for splittings}

The first step to searching for splittings over subgroups commensurable with $H$ is to be able to recognise such subgroups. We will achieve this by deciding membership of $\Comm_G(H)$. 

\begin{proposition}\label{prop:decide-fi}
Let $G$ be a hyperbolic group, then there is an algorithm which, on input of generators of a quasiconvex subgroup $H$, will decide if $|G:H| < \infty$.
\end{proposition}

\begin{proof}
We have that $|G : H|$ is finite if and only if $\partial G = \Lambda H$. By Proposition~\ref{prop:arrk-comp-corr} this is true if and only if $A_{r,R,K}(H) = \emptyset$ for suitably chosen $r, R, K$. This can be decided by computing a finite $H$-cover $F$ of $N_{r,R}(H)$ as in the proof of Proposition~\ref{prop:decide-arrk} and checking if $F$ intersects $C_K(H)$. 
\end{proof}

\begin{proposition}\label{prop:comm-membership}
Let $G$ be a hyperbolic group. Given generators of a quasiconvex subgroup $H \leq G$, then membership of the commensurator $\Comm_G(H)$ is decidable. 
\end{proposition}

\begin{proof}
Let $g \in G$, then since $G$ is hyperbolic we have that $H^g$ is quasiconvex. Moreover, $H \cap H^g$ is quasiconvex and we can compute an explicit generating set for this group via \cite{gitik2017intersections}. We then use Proposition~\ref{prop:decide-fi} to decide if $|H : H \cap H^g|$ and $|H^g : H \cap H^g|$ are finite. This decides whether $g \in \Comm_G(H)$. 
\end{proof}

Note that $\Comm_G(H)$ is itself quasiconvex, and so given a generating set of this subgroup we would have that the membership problem would be decidable via Kapovich's algorithm. However, we are not given generators of $\Comm_G(H)$, but of $H$. So, what the above proposition tells us is that we can decide membership of the commensurator in spite of this problem. 
Applying this observation, we produce the following algorithm which searches for splittings where the edge group is commensurable with $H$. 

\begin{proposition}\label{prop:search-for-splittings}
There is an algorithm which takes in as input a one-ended hyperbolic group $G$ and generators of a quasiconvex subgroup $H$, and terminates if and only if $H$ is associated with a splitting.
\end{proposition}

\begin{proof}
Enumerate presentations of $G$ via Tietze transformations. If a given presentation has the general form of an amalgam or HNN extension, run Kapovich's algorithm \cite{kapovich1996detecting} on the generators of the edge group, which terminates if and only if this subgroup is quasiconvex and outputs a quasiconvexity constant $Q$ if it does terminate. This procedure enumerates splittings of $G$ over quasiconvex subgroups. 

Given a particular quasiconvex splitting of $G$, say over $H'$, we can decide if $H'$ is commensurable with $H$ as follows. Using Proposition~\ref{prop:comm-membership} we decide if $H' \leq \Comm_G(H)$ and $H \leq \Comm_G(H')$. It is easy to check these two relations hold if and only if $H$ is commensurable with $H'$. This completes the algorithm. 
\end{proof}

\subsection{An algorithm to search for crossings}

Recall Proposition~\ref{prop:crossings-in-boundary}, which characterised crossings via intersections of shadows in $\partial G$. We now characterise these intersections via local geometry, and present an algorithm which terminates if and only if such a crossing exists. 
We first need the following technical lemma.

\begin{lemma}\label{lem:char-crossings}
Let $U_1$, $U_2$ be unions of connected components of $A_{r ,\infty, K}(H)$. 
The intersection $\shad U_1 \cap \shad (gU_2)$ is non-empty if and only if there exists some $x \in U_1 \cap gU_2$ such that 
$$
d(x, H) > K \ \ \ \ \textrm{and} \ \ \ \  d(x, gH) > K + 5\delta + |g|.
$$
\end{lemma}

\begin{proof}
Firstly, suppose that such an $x$ exists, then let $\gamma$ be a ray based at 1 passing within $C = 3\delta$ of $x$, as in Lemma~\ref{lem:visibility}. As in the proof of Lemma~\ref{lem:shad-nonempty} we see that $\gamma(\infty) \in \shad U_1$ since $\gamma$ passes within $C$ of $C_K(H)$. Secondly, let $\gamma'$ be a geodesic ray based at $g$ such that $\gamma'(\infty) = \gamma(\infty)$. Then the Hausdorff distance between $\gamma$ and $\gamma'$ is at most $5\delta + |g|$ (apply e.g. \cite[Exc.~11.86]{dructu2018geometric}). Again, as in the proof of Lemma~\ref{lem:shad-nonempty} we see that $\gamma'(\infty) \in \shad (gU_2)$, and we're done. 

Conversely, let $\gamma \in p \in \shad U_1 \cap \shad (gU_2)$. By the definition of $\shad$ we have that there is some $t_0$ such that for all $t \geq t_0$, $\gamma(t) \in U_1 \cap gU_2$. Moreover, by Lemma~\ref{lem:liminf} we have for $i = 1,2$ that $d(\gamma(t), H_i) \to \infty$ as $t \to \infty$. Thus, by setting $x = \gamma(t)$ for some sufficiently large $t$, we are done. 
\end{proof}

We're now ready to search for crossings. This algorithm will  check every choice of $H$-almost invariant set and search for any crossings. It will terminate if and only if it finds a crossing for every such choice. Note that since $e(G,H) < \infty$ by Corollary~\ref{cor:ends-of-pair-finite}, there is only finitely many possible $H$-almost invariant sets to check, up to equivalence.

\begin{proposition}\label{prop:search-for-crossings}
There exists an algorithm which, on input of a one-ended hyperbolic group $G$ and generators of a quasiconvex subgroup $H$, will terminate if and only if for every choice $X\subset G$ of $H$-almost invariant set, we have that $X$ is not semi-nested. 
\end{proposition}

\begin{proof}
We begin by picking representative choices for every equivalence class of non-trivial $H$-almost invariant subsets $X_1, \ldots, X_n$. In particular, each $X_i$ is a union of $H$-orbits of connected components of $A_{r,\infty,K}(H)$. For notational convenience we identify $X_i^\ast$ with its representative in this list. 

Enumerate elements $g \in G - \Comm_G(H)$ via Proposition~\ref{prop:comm-membership}. 
For each $i = 1, \ldots , n$ search for some $x_{i,1} \in X_i \cap gX_i$, $x_{i,2} \in X_i \cap g(X_i^\ast)$, $x_{i,3} \in X_i^\ast \cap gX_i$, and $x_{i,4} \in X_i^\ast \cap g(X_i^\ast)$ such that the conditions in Lemma~\ref{lem:char-crossings} are met for each $x_{i,j}$. 
We terminate our search if and only if we find such an $x_{i,j}$ for every $i$, $j$. By Proposition~\ref{prop:crossings-in-boundary} and Lemma~\ref{lem:char-crossings} this will terminate if and only if every $H$-almost invariant subset is not semi-nested. 
\end{proof}

\subsection{Splitting detection}\label{sec:spltting-algs}

We now present the final result of this paper, an algorithm to detect splittings over quasiconvex subgroups. We split this algorithm into two cases, and firstly we consider the situation that we know \textit{a priori} that our subgroup has finitely many filtered ends. 

\begin{theorem}\label{thm:decide-splitting-finitecoends}
There is an algorithm which, upon input of a one-ended hyperbolic group $G$ and generators of a quasiconvex subgroup $H$ such that $\tilde e(G,H) < \infty$, will decide whether $H$ is associated with a splitting. Furthermore, if such a splitting exists then the algorithm will output this splitting. 
\end{theorem}

\begin{proof}
Let $U_1, \ldots , U_n$ be the components of $A_{r,\infty, K}(H)$, and let $H' \leq H$ be a finite index subgroup which fixes each individual component, so each $U_i$ is an $H'$-almost invariant subset. Any other subgroup which is commensurable with $H$ will have the same set of filtered ends, so this is the ``finest'' set of $H''$-almost invariant subsets for any subgroup $H''$ commensurable with $H$. It therefore follows that if $H'$ does not admit a semi-nested $H'$-almost invariant set, then neither does any other subgroup which is commensurable with $H$. 

With the above in mind we run two algorithms in parallel. We search for a splitting over a subgroup commensurable with $H$ via Proposition~\ref{prop:search-for-splittings}, and concurrently run the algorithm in Proposition~\ref{prop:search-for-crossings} on $H'$. By the above discussion exactly one of these will terminate, and if the former algorithm terminates then it will output a presentation of a splitting over a subgroup commensurable with $H$. 
\end{proof}

Indeed, if $\tilde e(G,H)$ is not finite then we cannot rely on the machinery used above, as the stabiliser of some component of $\partial G - \Lambda H$ may have infinite index in $H$. Moreover, referring back to the discussion in Section~\ref{sec:counting-ends}, we are unlikely to be able to decide if $\tilde e (G,H)$ is finite for arbitrary quasiconvex $H$, at least with just the current tools presented in this paper. 

Recall that a subgroup $H$ in $G$ is said to be lonely if there is no subgroup $H' \neq H$ such that $H$ is commensurable to $H'$. For a quasiconvex subgroup $H$ of a hyperbolic group $G$, this condition is equivalent to saying that $H = \Comm_G(H)$ and $H$ has no finite quotients. 

\begin{theorem}\label{thm:decide-splittings-lonely}
There is an algorithm which, upon input of a one-ended hyperbolic group $G$, generators of a quasiconvex subgroup $H$, and knowledge of whether $H$ is lonely in $G$, will decide whether $H$ is associated with a splitting. Furthermore, if such a splitting exists then the algorithm will output this splitting. 
\end{theorem}

\begin{proof}
Firstly, begin running the algorithm from Proposition~\ref{prop:search-for-splittings}, which will terminate if and only if $H$ is associated with a splitting.

Concurrently we run the following. If $H$ is not lonely, then simultaneously search for an element $g \in \Comm_G(H) - H$ and a finite index subgroup $H'$ of $H$. If we find the former then continue, and if we find the latter then replace $H$ with $H'$ and then continue. At least one of these will terminate, and this ensures that $H \neq \Comm_G(H)$. If $H$ is lonely, then just continue.
We now run the algorithm presented in Proposition~\ref{prop:search-for-crossings}. By Proposition~\ref{prop:quasiconvex-splittings}, exactly one of these two procedures will terminate, and if the first algorithm terminates then it will output a presentation of a splitting over a subgroup commensurable with $H$. 
\end{proof}

\begin{corollary}\label{cor:detect-split-rf}
    There is an algorithm which takes in as input a one-ended hyperbolic group $G$ and generators of a quasiconvex, residually finite subgroup $H$. 
    This algorithm will then decide if $H$ is associated with a splitting, and will output such a splitting if one exists. 
\end{corollary}

It is conjectured that the problem of deciding whether a given hyperbolic group has a finite quotient is undecidable, and in fact this problem is known to be equivalent to the well-known conjecture that there exists a hyperbolic group which is not residually finite \cite{bridson2015triviality}. Assuming this conjecture, it would be undecidable whether a given quasiconvex subgroup is lonely. This means that the hypothesis in Theorem~\ref{thm:decide-splittings-lonely} is likely necessary unless we place further restrictions on $H$, such as requiring $H$ be residually finite. 

% We conclude by noting that one can easily drop the condition that $G$ be one-ended. Indeed, if $H$ is finite then it is associated to a splitting if and only if $e(G) \geq 1$. Otherwise, we may apply \cite{dahmani2008detecting} to a maximal finite splitting of $G$, which necessarily exists due to the accessibility result of Dunwoody \cite{dunwoody1985accessibility}. Then, $H$ is associated to a splitting of $G$ if and only if $H$ is associated to a splitting of some one-ended vertex group of this maximal finite splitting, up to conjugacy.

\bibliographystyle{abbrv}
\bibliography{references}

% \appendix

% \section{appendix?}\label{sec:appendix}

\end{document}